\documentclass[a4paper,10pt,reqno]{amsart}
\usepackage{bm,amsmath,amsfonts,amscd,amsthm,amssymb,graphicx,mathrsfs,eufrak,mathrsfs}
\usepackage{tikz-cd}

\usepackage{enumitem}

\usepackage[nospace,noadjust]{cite}

\usepackage[dvips,all,arc,curve,color,frame]{xy}
\usepackage{tikz}
\usetikzlibrary{arrows,decorations.pathmorphing,decorations.pathreplacing,positioning,shapes.geometric,shapes.misc,decorations.markings,decorations.fractals,calc,patterns}
\usepackage{tikz-cd}
\usepackage[colorlinks]{hyperref}
\usepackage{comment}
\usepackage{mathtools}
\usepackage{mathdots}
\usepackage{todonotes}
\usepackage{stmaryrd}

\tikzset{>=stealth',
        cvertex/.style={circle,draw=black,inner sep=1pt,outer sep=3pt},
        vertex/.style={circle,fill=black,inner sep=1pt,outer sep=3pt},
        star/.style={circle,fill=yellow,inner sep=0.75pt,outer sep=0.75pt},
        tvertex/.style={inner sep=1pt,font=\scriptsize},
        gap/.style={inner sep=0.5pt,fill=white}}

\author{Martin Kalck} 
\email{martin.maths@posteo.de}
\urladdr{}

\author{Joseph Karmazyn} 
\address{School of Mathematics and Statistics,
University of Sheffield,
Hicks Building,
Hounsfield Road,
Sheffield,
S3 7RH.}
\email{j.h.karmazyn@sheffield.ac.uk}
\urladdr{http://www.jhkarmazyn.staff.shef.ac.uk/}

\date{\today}

\title[Ringel duality]{Ringel duality \\ for certain strongly quasi-hereditary algebras}


\DeclareMathAlphabet{\mathbbm}{U}{bbm}{m}{n}

\newcommand{\Hom}{\textnormal{Hom}}

\newcommand{\Ext}{\textnormal{Ext}}

\newcommand{\Coh}{\textnormal{Coh}}

\newcommand{\QCoh}{\textnormal{QCoh}}
\newcommand{\op}{\textnormal{op}}

\newcommand{\End}{\textnormal{End}}

\newcommand{\proj}{\textnormal{proj}}

\def\mod{\mathop{\textnormal{mod}}\nolimits}

\let\oldtocsection=\tocsection
\let\oldtocsubsection=\tocsubsection
\let\oldtocsubsubsection=\tocsubsubsection
\renewcommand{\tocsection}[2]{\hspace{0em}\oldtocsection{#1}{#2}}
\renewcommand{\tocsubsection}[2]{\hspace{1em}\oldtocsubsection{#1}{#2}}
\renewcommand{\tocsubsubsection}[2]{\hspace{2em}\oldtocsubsubsection{#1}{#2}}

\theoremstyle{plain}
\newtheorem{thm}{Theorem}[section]
\newtheorem{cor}[thm]{Corollary}
\newtheorem{lem}[thm]{Lemma}
\newtheorem{prop}[thm]{Proposition}

\newtheorem*{Preprop}{Proposition}
\newtheorem*{Precor}{Corollary}

\theoremstyle{definition}

\newtheorem{defn}[thm]{Definition}
\newtheorem{ex}[thm]{Example}

\theoremstyle{remark}
\newtheorem{rem}[thm]{Remark}
\theoremstyle{setup}

\newcommand{\AUS}{\mathsf{AUS}}
\newcommand{\PI}{\mathsf{PI}}
\newcommand{\SUB}{\mathsf{SUB}}
\newcommand{\sub}{\mathsf{sub}}
\newcommand{\FAC}{\mathsf{FAC}}
\newcommand{\fac}{\mathsf{fac}}
\newcommand{\ADR}{\mathsf{ADR}}
\newcommand{\adr}{\mathsf{adr}}
\DeclareMathOperator{\rad}{\mathsf{rad}}
\DeclareMathOperator{\coker}{\mathsf{coker}}

\newcommand{\ind}{\operatorname{ind}}
\newcommand{\im}{\mathsf{im}}

\DeclareMathOperator{\add}{\mathsf{add}}

\newcommand{\ca}{{\mathcal A}}

\newcommand{\cs}{{\mathcal S}}

\begin{document}

\begin{abstract}

We study quasi-hereditary endomorphism algebras defined over a new class of finite dimensional monomial algebras with a special ideal structure. The main result is a uniform formula describing the Ringel duals of these quasi-hereditary algebras.

As special cases, we obtain a Ringel-duality formula for a family of strongly quasi-hereditary algebras arising from a type A configuration of projective lines in a rational, projective surface as recently introduced by Hille and Ploog, for certain Auslander-Dlab-Ringel algebras, and for Eiriksson and Sauter's nilpotent quiver algebras when the quiver has no sinks and no sources. We also recover Tan's result that the Auslander algebras of self-injective Nakayama algebras are Ringel self-dual.

 \end{abstract}
 
\maketitle

\tableofcontents

\maketitle

\section{Introduction}

Quasi-hereditary algebras form an important class of finite dimensional algebras with relations to Lie theory (this was the original motivation \cite{Scott}) and exceptional sequences in algebraic geometry (see e.g. \cite{HillePerling} and \cite{BuchweitzLeuschkeVdB}). Examples of quasi-hereditary algebras include blocks of category $\mathcal{O}$ and Schur algebras.

Ringel duality \cite{Ringel} is a fundamental phenomenon in the theory of quasi-hereditary algebras, see for example \cite{KrauseComp, Rouquieretal, BodzentaKuelshammer, CoulembierMazorchuk, Pucinskaite, IyamaReiten2Auslander, ErdmannParker, CondeErdmann,Coulembier} for (recent) work on this topic. For any quasi-hereditary algebra $A$ there exists a quasi-hereditary algebra $\mathfrak{R}(A)$,  the \emph{Ringel-dual} of $A$, such that
\[
A \textnormal{-mod} \cong \mathfrak{R}(\mathfrak{R}(A))\textnormal{-mod}.
\]However, computing the Ringel-dual of a quasi-hereditary algebra explicitly may not be straightforward. In this paper we introduce a new class of quasi-hereditary algebras that admit a uniform description of their Ringel duals, see Theorem \ref{T:PrelimMain}.

Let's make this more precise. Let $k$ be an algebraically closed field, and $R$ be a finite dimensional monomial $k$-algebra, i.e. $R=kQ/I$, where $I$ is a two-sided ideal generated by paths in $Q$. For example $R=k\langle x_1, \ldots, x_l \rangle/I$, where $I$ is a two-sided ideal generated by monomials in $k\langle x_1, \ldots, x_l \rangle$. 

\begin{defn}
We call $R$ \emph{ideally ordered}, if for every primitive idempotent $e \in R$ and every pair of monomials $m, n \in eR$ there exists an epimorphism $Rm \to Rn$ or an epimorphism $Rn \to Rm$.
\end{defn}

For an algebra $R$ we consider the additive subcategory of all torsionless $R$-modules
\[
\sub(R):=\add \{U \mid U \subseteq R^{\oplus n} \} \subseteq R-\mod,
\]
define $\SUB(R):=\bigoplus_{U \in \ind(\sub(R))} U$ to be the direct sum of all indecomposable modules in $\sub(R)$ up to isomorphism, and set 
\[
E_R:=\End_{R}(\SUB(R)).
\]
For submodules $\Lambda \subset R$ we define the layer function $l(\Lambda):=\dim_k R - \dim_k \Lambda$ and we call $l$ the \emph{ideal layer function}. For an ideally ordered algebra $R$ the isomorphism classes of submodules $\Lambda \subset R$ label the simple modules $S(\Lambda)$ of $E_R$ and so the ideal layer function induces a partial ordering on the simple $E_R$-modules: $S(\Lambda_1) \le S(\Lambda_2) \Leftrightarrow l(\Lambda_1)  \le l(\Lambda_2)$. We call this the \emph{ideal layer ordering}.

The following is the main result of this paper and calculates the Ringel dual for algebras of the form $E_R$. See Theorem \ref{T:Main} for a more detailed version. 
 
\begin{thm}\label{T:PrelimMain} Let $R$ be a finite dimensional ideally ordered monomial algebra. Then $E_R$ is quasi-hereditary with respect to the ideal layer ordering, has global dimension $\le 2$, and has Ringel dual $E_{R^{\op}}^{\op}$:
\begin{align*} \label{E:PrelimRingelDualityFormula}
\mathfrak{R}(E_R) \cong E_{R^{\op}}^{\op}.
\end{align*}
\end{thm}

\begin{rem}
As we were preparing to post this paper on the ArXiv we became aware of the very recent paper \cite{Coulembier} of Coulembier that had just appeared. This paper introduces a more general version of the Auslander-Dlab-Ringel construction and proves a Ringel duality formula in this setting. In particular, this generalises the Ringel duality formula of Conde and Erdmann \cite{CondeErdmann} that we discuss below.

Our construction appears to be a special case that fits into this more general framework which, in particular, implies the Ringel duality formula of Theorem \ref{T:PrelimMain}. However, the approach and proof in Coulembier's work is different to the one in this paper. The work of Coulembier also seems to answer the questions we raise in Remark \ref{rem:Main Theorem} (1) and at the end of Section \ref{Section:ADR} regarding the possibility of finding a more general framework in which a Ringel duality formula holds. 

  In light of this, the results of this paper can be thought of as providing a very explicit example of Coulembier's Ringel duality formula, linking to several geometrically inspired examples such as Kn\"{o}rrer invariant algebras, and proving further properties that hold in our special case of the algebras $E_R$ such as being simultaneous left and right strongly quasi-hereditary for the same quasi-hereditary order and being left ultra strongly quasi-hereditary.

\end{rem}

The class of ideally ordered monomial algebras includes many well known examples, and in many of these examples the endomorphism algebras $E_{R}$ are also well understood.

\begin{ex} \label{ex:ideallyordered} The following families of finite dimensional monomial algebras are ideally ordered.
\begin{enumerate}[start=0]
\item Hereditary algebras.
\item The algebras $R=k\langle x_1, \ldots, x_l \rangle/(x_1, \ldots, x_l)^m$ for positive integers $l, m$. 

\item More generally, for $Q$ a finite quiver, $J \subseteq kQ$ the two-sided ideal generated by all arrows in $Q$, and $m \ge 0$ the algebra $R:= kQ/J^m$ is ideally ordered. 

To prove this, consider a monomial $p \in eR$. There is a surjection $Re \rightarrow Rp$ given by $g \mapsto gp$ with kernel
\[
\{ g \in Re \, | \, gp =0 \} \cong \{ g \in kQ \, | \, gp \in J^m \} \cong J^{m-i}e
\]
where $i$ is minimal such that $p \in J^i$. Hence for any monomial $p \in eR$ there is an isomorphism $Rp \cong Re/J^le$ for some $l \in \{ 1, \ldots, m \}$. As a result, for any pair of monomials $p,q \in eR$ the monomial ideals $Rp, Rq$ are isomorphic to some pair of quotient modules occurring in the chain of surjections 
\[
Re \cong Re/J^{m}e \to Re/J^{m-1}e \to \ldots \to Re/J^{1}e.
\]
Hence there is a surjection $Rp \rightarrow Rq$ or $Rq \rightarrow Rp$.

\item

For every pair $0<a<r$ of coprime integers the finite dimensional monomial \emph{Kn{\"o}rrer invariant algebra} $K_{r,a}$ is defined in \cite[Definition 4.6]{KalckKarmazyn}, and the results of \cite[Section 6.4]{KalckKarmazyn} describe its monomial ideals and imply that it is ideally ordered. The definition of these algebras is recapped in Section \ref{Section:HillePloog}.

\item Nakayama algebras, introduced in \cite{Nakayama}, are ideally ordered. 
\end{enumerate}
We give two constructions that can be used to produce ideally ordered monomial algebras.
\begin{enumerate}[resume]
\item Let $R$ and $K$ be ideally ordered monomial algebras and let $_R M_K$ be an $R$-$K$-bimodule which is projective as  $R$-module and as $K$-module. Then 
\[T:=\begin{pmatrix} R & _R M_K \\ 0 & K  \end{pmatrix} \] is an ideally ordered monomial algebra. Example \ref{E:Goodleftapproximation} (a) shows that $T$ need not be ideally ordered if we weaken the assumptions on $_R M_K$.

\item If $R$ is ideally ordered and $e \in R$ is an arbitrary idempotent, then $eRe$ is ideally ordered. 

Suppose that $f \in eRe$ is a primitive idempotent and $p,q \in feRe=fRe$ are monomials. Then $f$ is a primitive idempotent in $R$, $p,q \in fR$ are monomials, and as $R$ is ideally ordered there is a surjection between $Rp$ and $Rq$.  Applying
\[
eR \otimes_R (-) \cong \Hom_R(Re,-):R\text{-}\mod \rightarrow eRe\text{-}\mod
\]
to this surjection of $R$-modules will produce the required surjection of $eRe$-modules between $eRp$ and $eRq$ since $eR \otimes_R (-)$ is exact. This shows $eRe$ is ideally ordered.

\end{enumerate}
We finish by exhibiting a local commutative monomial algebra which is not ideally ordered.
\begin{enumerate}[resume]
\item The algebra $R=k[x, y]/(x^3, xy, y^3)$ is not ideally ordered. To see this consider the ideals $Rx$ and $Ry$.
\end{enumerate}

\end{ex}
We briefly discuss how these examples of ideally ordered monomial algebras $R$, and the algebras $E_R:=\End_R(\SUB(R))$ they define, relate to algebras and results in the literature.

\subsection*{Hille and Ploog's algebras}
The Ringel duality formula of Theorem \ref{T:PrelimMain}, the definition of ideally ordered monomial algebras, and the construction of the algebras $E_R$ in this paper are all geometrically inspired. They were first observed in our previous work \cite{KalckKarmazyn} for a class of quasi-hereditary algebras $\Lambda_\alpha$ constructed by Hille \& Ploog  \cite{HillePloog}. 

In more detail, the algebras $\Lambda_{\bm{\alpha}}$ arise  from an exceptional collection of line bundles associated to a type $A$ configuration of intersecting rational curves $C_i$ in a rational, projective surface as illustrated in the picture below.

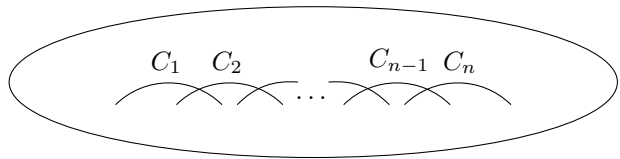
\[
\begin{tikzpicture}
\draw (-2.6,-0.3) to [bend left=45] node[pos=0.48, above] {$C_1$} (-1.2,-0.3);
\draw (-1.8,-0.3) to [bend left=45] node[pos=0.48, above] {$C_2$} (-0.4,-0.3);
\draw (-1,-0.3) to [bend left=25] node[pos=0.48, right] {} (-0.2,0);
\node (D) at (0,-0.2) {$\dots$};

\draw (2.6,-0.3) to [bend right=45] node[pos=0.48, above] {$C_n$} (1.2,-0.3);
\draw (1.8,-0.3) to [bend right=45] node[pos=0.48, above] {$C_{n-1}$} (0.4,-0.3);
\draw (1,-0.3) to [bend right=25] node[pos=0.48, right] {} (0.2,0);

\draw[black] (0,0) ellipse (4 and 1);
\end{tikzpicture}
\]
The construction of $\Lambda_{\bm{\alpha}}$ (recapped in Section \ref{Section:HillePloog}) depends on the order of the curves $C_i$.
Reversing the order of these curves, Hille \& Ploog's construction yields an algebra $\Lambda_{\bm{\alpha}^{\vee}}$.

It is natural to ask how the algebras $\Lambda_{\bm{\alpha}}$ and $\Lambda_{\bm{\alpha}^{\vee}}$ are related from a representation theoretic perspective. Our answer below is phrased in terms of Ringel duality.

\begin{Preprop}[See Proposition \ref{Prop:RingelforHillePloog}] There is an isomorphism of algebras
\begin{align}\label{E:Rdualityformula}
\mathfrak{R}(\Lambda_{\bm{\alpha}}) \cong  \Lambda_{\bm{\alpha}^{\vee}}^{\op}.
\end{align}
\end{Preprop}
In order to see that \eqref{E:Rdualityformula} is a special case of our main Theorem \ref{T:PrelimMain}, we recall that there are isomorphisms of algebras
\begin{align*}
\Lambda_{\bm{\alpha}} \cong E_{K_{r,a}} \text{ and }\Lambda_{\bm{\alpha}^\vee} \cong E_{K_{r,a}^\op} 
\end{align*}
described in \cite[Section 6]{KalckKarmazyn}. This is recalled in Proposition \ref{P:HPisKnorrer} and the discussion immediately beneath it. Here, $K_{r, a}$ denotes a Kn{\"o}rrer invariant algebra, which is the ideally ordered monomial in Example \ref{ex:ideallyordered} (3), and $0 < a < r$ are a pair of coprime integers depending on $\bm{\alpha}$.

We remark that in this setting the Ringel duality formula \eqref{E:Rdualityformula} also has an alternative proof, which is more geometric, see Proposition \ref{Prop:RingelforHillePloog}.

The aim of this paper was to find a more general representation theoretic framework extending the Ringel-duality formula \eqref{E:Rdualityformula} to a larger class of (ultra) strongly quasi-hereditary algebras. In particular, the Kn{\"o}rrer invariant algebras are the original motivation for the ideally ordered condition.

\begin{rem}
The algebras $\Lambda_{\bm{\alpha}} \cong E_{K_{r, a}}$ and $K_{r, a}$ were used to show a noncommutative version of Kn\"{o}rrer periodicity for cyclic quotient surface singularities in \cite{KalckKarmazyn}. More precisely, it was proved there that the singularity category of a cyclic quotient surface singularity is equivalent to the singularity category of a corresponding Kn\"{o}rrer invariant algebra, generalising classical Kn\"{o}rrer's periodicity for the polynomials $x^n$ and $x^n + y^2 +z^2$. The proof uses noncommutative resolutions and $\Lambda_{\bm{\alpha}}\cong E_{K_{r,a}}$ plays the role of a noncommutative resolution for $K_{r,a}$.
\end{rem}

\subsection*{Auslander-Dlab-Ringel and nilpotent quiver algebras}

From a more representation theoretic viewpoint, a Ringel duality formula that looks similar to that of Theorem \ref{T:PrelimMain} was proved for Auslander-Dlab-Ringel algebras $E_R^{\ADR}$ by Conde and Erdmann \cite[Theorem A]{CondeErdmann}. We define these algebras, recall Conde and Erdmann's Ringel duality formula, and discuss the relationship between this result and the results of this paper in Section \ref{Section:ADR}.

In particular, for the class of algebras $R:= kQ/J^m$ in Example \ref{ex:ideallyordered} $(2)$ above the corresponding algebras $E_R$ and $E_R^{\ADR}$ coincide if $Q$ has no sources. 

\begin{Preprop} [See Proposition \ref{P:ADR=IOM}]
If $R:= kQ/J^m$ for $Q$ a finite quiver without sources and $J$ the two-sided ideal generated by all arrows in $Q$, then there is an isomorphism of quasi-hereditary algebras
\[
E_R^{\ADR} \cong E_R.
\]
\end{Preprop}

We also prove that when $Q$ has no sinks the ADR algebra coincides with the quiver nilpotent algebra $N_m(Q)$ introduced by Eiriksson and Sauter \cite{EirikssonSauter}, which is motivated via a quiver graded version of Richardson orbits and is recapped in Section \ref{Section:NQA}.

\begin{Preprop} [See Proposition \ref{P:ADR and NSQ}]
If $R:= kQ/J^m$ for $Q$ a finite quiver without sinks and $J$ the two-sided ideal generated by all arrows in $Q$, then there is an isomorphism of quasi-hereditary algebras 
\[
N_{m}(Q) \cong E_{kQ/J^m}^{\ADR}.
\]
\end{Preprop}

In particular, if $R=kQ/J^m$ (as in Example \ref{ex:ideallyordered} (2) above) for a quiver with no sinks or sources, then $E_R \cong E_R^{\ADR} \cong N_m(Q)$ and so Theorem \ref{T:PrelimMain} provides a Ringel duality formula for such nilpotent quiver algebras; see Corollary \ref{C:Qnosinksnosources}.

\subsection*{Nakayama and Auslander algebras.}

Several of the examples of ideally ordered monomial algebras above can be thought of as geometrically inspired by resolutions of singularities. Indeed, Examples \ref{ex:ideallyordered} (1) -- (4) above can be thought of as different generalisations of the algebra $k[x]/x^n$.

Work of Dlab \& Ringel \cite{DR2} shows that every finite dimensional algebra admits a noncommutative `resolution' by a quasi-hereditary algebra, and a generalisation of this result led to Iyama's proof of the finiteness of Auslander's representation dimension \cite{Iyama}. 

Such a resolution for finite dimensional algebras of finite representation type is provided by the Auslander algebra. This also occurs in more geometric contexts; the categorical resolutions considered by Kuznetsov and Lunts \cite{KuznetsovLunts} use a construction motivated by Auslander algebras to resolve non-reduced schemes.

For $R$ a finite dimensional algebra of finite representation type let $E_R^{\AUS}$ denote the Auslander algebra of $R$, which we recall in Section \ref{Section:AuslanderNakayama}.
\begin{Preprop}[See Proposition \ref{p:AUSisIOM}] If $R$ is an ideally ordered monomial algebra, then
\[
E_R^{\AUS} \cong E_R
\]
if and only if $R$ is self-injective.
\end{Preprop}

A particular example of a class of ideally ordered, monomial algebras of finite representation type are the Nakayama algebras (listed as Example (4) above).

\begin{Precor}[See Corollary \ref{c:AUSisNAK}] If $R$ is self-injective Nakayama algebra, then
\[
E_R^{\AUS} \cong E_R.
\]
\end{Precor}

In this setting Theorem \ref{T:PrelimMain} also generalises several known results in the literature, e.g. that the Auslander-algebras of selfinjective Nakayama algebras are Ringel-selfdual, see \cite{Tan}.

\begin{Precor}[See Corollary \ref{c:RingelSelfDual}]
If $R$ is a self-injective Nakayama algebra then 
\[
\mathfrak{R}(E_R) \cong E_R.
\]
\end{Precor}

For other related results see work by Baur, Erdmann \& Parker \cite{ErdmannBaurParker}, Crawley-Boevey \& Sauter \cite{CrawleyBoeveySauter}, and Nguyen, Reiten, Todorov \& Zhu \cite{VanReitenTodorovShi}).

\subsection*{Left and right strongly quasi-hereditary structure.}

A further special property of the quasi-hereditary algebras $E_R=\End_R(\SUB(R))$ is that the ideal layer function simultaneously realises both a left and right strongly quasi-hereditary structure on the algebras. 

Since $\add \SUB(R)$ is closed under kernels $E_R=\End_R(\SUB(R))$ has global dimension $2$, and it was recently shown by Tsukamoto \cite{Tsukamoto} that this implies $E_R$ admits both a left strongly quasi-hereditary structure and a right strongly quasi-hereditary structure (for a possibly different order),  building on earlier work of Dlab \& Ringel and Iyama. 

In general the left and right strongly quasi-hereditary structures cannot be realised using the same order. Indeed, Tsukamoto shows that for Auslander-algebras of representation-finite algebras (which all have global dimension $2$) this is possible precisely if the underlying algebra is a Nakayama algebra. 

As seen in the examples above, the class of quasi-hereditary algebras $E_R$ constructed from ideally ordered monomial algebras provides a larger class of such algebras.

\subsection*{Conventions}
Throughout this paper $k$ will denote an algebraically closed field. For paths $p,q \in kQ$ in the path algebra of a quiver $Q$ the composition $pq$ will denote the path $q$ followed by the path $p$. For $R$ a Noetherian ring $R$-$\mod$ will denote the category of finitely generated left $R$-modules, and for $S \subset R$-$\mod$ we will define $\add S$ to be the additive subcategory generated by $S$: i.e. the smallest full subcategory of $R$-$\mod$ containing $S$ and closed under isomorphism, direct sums, and direct summands. In particular, the category of finitely generated projective $R$ modules $\proj$-$R$ is equivalent to $\add R$.

We recall the category of torsionless $R$-modules $\sub(R)$ from the introduction, and now give a more general definition: for an $R$-module $M$ we define the following subcategory 
\[
\sub(M):= \add \{U \mid U \subseteq M^{\oplus n} \} \subseteq R-\mod
\]
 with corresponding module $\SUB(M):=\bigoplus_{U \in \ind(\sub(M))} U$.

Moreover, for an $R$-module $M$, we set
\[
\fac(M):=\add \{Q \mid M^{\oplus n} \to Q \to 0  \} \subseteq R-\mod
\]
and let $\FAC(M):=\bigoplus_{Q \in \ind(\fac(M))} Q$ denote the direct sum of all indecomposable objects in $\fac(M)$ up to isomorphism.

 We let $\dagger$ denote the standard $k$-duality $\Hom_k(-, k)$. For the injective cogenerator $I:=R_R^{\dagger}$ we define the category of \emph{divisible} $R$-modules
\[
\fac(I):=\add \{Q \mid I^{\oplus n} \to Q \to 0  \} \subseteq R-\mod
\]
and let $\FAC(I):=\bigoplus_{Q \in \ind(\fac(I))} Q$ denote the direct sum of all indecomposable objects in $\fac(I)$ up to isomorphism.

\subsection*{Acknowledgements} We thank Teresa Conde for interesting discussions about this work and about relations to her thesis. We are grateful to Karin Erdmann for pointing us to Ringel's paper which simplifies the proof of our main result and adds another perspective to this work. We also thank Agnieszka Bodzenta who, in particular, explained to us the proof of Proposition \ref{Prop:RingelforHillePloog} and Xiao-Wu Chen who shared with us Example \ref{ex:Not all Sub are PI}. We would also like to thank \"{O}gmunder Eiriksson, Julian K{\"u}lshammer, Daiva Pu{\v c}inskait{\. e}, {\v S}pela {\v S}penko, and Michael Wemyss for interesting and helpful discussions and David Ploog for pointing out misprints in an earlier version.

We would also particularly like to thank the anonymous referee, whose many useful comments have greatly improved the paper.

The first author was supported by EPSRC grant EP/L017962/1. The second author was supported by  EPSRC grant EP/M017516/2.

\section{Strongly quasi-hereditary algebras}
In this section, we will give necessary and sufficient conditions for certain endomorphism algebras over ideally ordered monomial algebras to be left or right strongly quasi-hereditary.

We first recall the definition of a quasi-hereditary algebra. This needs some preparation. For a finite dimensional $k$-algebra $A$ choose a labelling $i \in I$ of the simple $A$-modules $S_i$ up to isomorphism. A  partial order $\leq$ on the set $I$ is called \emph{adapted} if for each $M \in A$-$\mod$ with top $S_i$ and socle $S_j$ incomparable there exists some $k>i$ or $k>j$ such that $S_k$ is a composition factor of $M$.  In particular, total orderings are adapted. We denote the projective cover and injective envelope of the simple $S_i$ by $P_i$ and $Q_i$ respectively.
\begin{defn}
Given a partial ordering $ \le$ on the index set $I$, for $i \in I$ the \emph{standard} module  $\Delta_i$ is the maximal factor module of $P_i$ whose composition series consists only of simple modules $S_j$ such that $j \le i$. Similarly, the \emph{costandard} module $\nabla_i$ is the maximal submodule of $Q_i$ whose composition series consists only of simple modules $S_j$ such that $j \le i$.

The $k$-algebra $A$ is \emph{quasi-hereditary} with respect to an adapted partial ordering $\le$ if:
\begin{enumerate}
\item
$\End_A(\Delta_i) \cong k$ for each $i \in I$ and 
\item
${A}$ can be filtered by the standard modules under this ordering; i.e. there exists a series of $A$-modules $0=M_n \subset M_{n-1} \subset \dots \subset M_1 \subset M_0=A$ such that each quotient $M_{i-1}/M_i$ is isomorphic to a direct sum of standard modules.
\end{enumerate}
\end{defn}

The following terminology is due to Ringel \cite{Iyamasfiniteness}. We refer to the references and discussions in \cite{Iyamasfiniteness} for earlier work.

\begin{defn} \label{D:Stronglyqh}
A quasi-hereditary algebra $A$ is called \emph{left strongly quasi-hereditary} if all standard modules have projective dimension at most $1$. It is called \emph{right strongly quasi-hereditary} if all costandard $A$-modules have injective dimension at most $1$.
\end{defn}

This is an equivalent characterisation of left/right strongly quasi-hereditary condition given in \cite[Appendix A1]{Iyamasfiniteness}. The original definition, introduced in \cite[Section 4]{Iyamasfiniteness}, is in terms of a layer function. 
\begin{defn} \label{D:RingelStrongQH}
A $k$-algebra $A$ is \emph{left strongly quasi-hereditary} with $n$ layers if there is a layer function $L\colon \{ \text{simple $A$-modules} \}{/\cong} \, \rightarrow \{1, \dots, n \}$ such that for any simple module $s$ with projective cover $P(s)$ there is an exact sequence
\[
0 \rightarrow R(s) \rightarrow P(s) \rightarrow \Delta(s) \rightarrow 0
\]
such that
\begin{itemize}
\item[a)]
The module $R(s)$ is the direct sum of projective covers $P(s')$ of simple modules $s'$ such that $L(s')>L(s)$.
\item[b)]
All simple factors $s'$ of $\rad \, \Delta(s)$ satisfy $L(s')<L(s)$.
\end{itemize}
The layer function induces an ordering on the simple {$A$}-modules  {and} the modules $\Delta(s)$ are the standard modules for this strongly quasi-hereditary structure. Right strongly quasi-hereditary algebras are defined dually.
\end{defn}

After some preparation, we introduce the class of endomorphism algebras which we are interested in. For the rest of this section we let $R$ be a finite dimensional $k$-{algebra}. A submodule of the form $Rp \subset R$ is a \emph{principal left ideal} if $p \in eR$ with $e \in R$ a primitive idempotent. We introduce the additive subcategory
\[
\mathsf{pi}(R):= \add \{ Rp \mid p \in eR, \text{ $e$ primitive idempotent} \} \subset R -\mod,
\]
and we let $\PI(R):= \bigoplus_{Rp \in \textnormal{ind}( \mathsf{pi}(R))} Rp$
denote the direct sum of all principal left ideals up to isomorphism. In this section we assume that $\PI(R)$ is finitely generated and define $E_R^{\PI}:=\End_R(\PI(R))$. 

The assumption on $\PI(R)$ is satisfied for ideally ordered monomial algebras $R$ due to Lemma \ref{L:PIareMonom} but does not hold for all finite dimensional algebras; e.g. if $R=\mathbb{C}[x,y]/(x^2,y^2)$, then the ideals $I_{\lambda}:=R(x+\lambda y)$ for $\lambda \in \mathbb{C}$ give a $\mathbb{C}$-indexed set of ideals that are pairwise non-isomorphic as left modules.

Throughout the rest of the paper we will label the simple and projective $E_R^{\PI}$-modules by the principal ideals of $R$, as we now explain. To do this we use the additive anti-equivalence
\begin{align} \label{E:AntiEquivalence}
\add \PI(R) \xrightarrow{\Hom_R(-, \PI(R))} E_R^{\PI} \text{-} \proj.
\end{align}
It is clear that $\Hom_R(-,\PI(R))$ is a contravariant functor, and one can show that it is an additive anti-equivalence using that it maps the additive generator $\PI(R)$ of $\add \PI(R)$ to the additive generator $E_R^\PI$ of $E_R^{\PI} \text{-} \proj$. Under this anti-equivalence the indecomposable summands $\Lambda$ of $\PI(R)$ are in 1-to-1 correspondence with indecomposable projective $E_R^{\PI}$-modules, which we denote by $P(\Lambda)$. The indecomposable projective modules $P(\Lambda)$ are in 1-to-1 correspondence with simple $E_R^{\PI}$-modules $S(\Lambda)$ that occur as their heads (i.e, so that $P(\Lambda) \rightarrow S(\Lambda)$ is a projective cover). Hence the principal ideals $\Lambda \subset R$ index the simple modules $S(\Lambda)$ of $E_R^\PI$. When given a partial ordering on the principal ideals, we use similar notation to label standard $\Delta(\Lambda)$ and costandard $\nabla(\Lambda)$ objects.

This labelling allows to define the following layer function for the algebra $E_R^{\PI}$.

\begin{defn} \label{D:ideallayerfunc}
Let $R$ be a finite dimensional algebra. For principal left $R$-ideals $\Lambda$, we define $l(S(\Lambda)):=l(\Lambda):= \dim_k R - \dim_k \Lambda$ and we call $l$ the \emph{ideal layer function}. It induces a partial ordering on the principal left $R$-ideals, which we call the \emph{ideal ordering}.
\end{defn}

We will now determine when the ideal layer function induces a left or right strongly quasi-hereditary structure on $E_R^{\PI}$ by considering left and right minimal approximations with respect to the ideal ordering. 

The notion of minimal approximation is common in representation theory; see \cite{KrauseSaorin} for a survey. A morphism $\alpha: \Gamma \rightarrow \Lambda$ is a left approximation for a class of modules $\mathcal{C}$ if $\Lambda \in \mathcal{C}$ and the induced morphism $\Hom_R(\Lambda,C) \rightarrow \Hom_R(\Gamma,C)$ is surjective for all $C \in \mathcal{C}$. A morphism $\Gamma \xrightarrow{\alpha} \Lambda$ is left minimal if any endomorphism $\phi$ of $\Lambda$ satisfying $\phi \circ \alpha = \alpha$ is an isomorphism. In particular, left minimal approximations are unique up to isomorphism.

Denote by $\mathsf{pi}(R)_{>i} \subseteq \mathsf{pi}(R)$ the full subcategory of direct sums of principal left $R$-ideals $\Lambda$ with $l(\Lambda)>i$.

\begin{lem}
Let $\Gamma$ be a principal left ideal of layer $\gamma$. There is a minimal left $\mathsf{pi}(R)_{>\gamma}$  approximation $\alpha_\Gamma \colon \Gamma \to \Gamma_{>\gamma}$ of $\Gamma$.
\end{lem}
\begin{proof}
It is well-known that $\Gamma$ admits a left $\mathsf{pi}(R)_{>\gamma}$ approximation $\Phi \colon \Gamma \to \Lambda$. Indeed, this follows
since there are only finitely many indecomposable objects in $\mathsf{pi}(R)_{>\gamma} \subseteq R-\mod$ and since $R$ is finite dimensional, see e.g. \cite{AuslanderSmalo}. For the convenience of the reader, we recall the argument. We consider the module \[
\Lambda:= \bigoplus_{M \in \mathsf{pi}(R)_{> \gamma}} M^{\oplus  \dim \Hom_R(\Gamma,M)}
\]
where the sum is taken over all indecomposable objects $M$ in $\mathsf{pi}(R)_{> \gamma}$ (up to isomorphism). Then $\Lambda \in \mathsf{pi}(R)_{> \gamma}$ as each $\Hom_R(\Gamma,M)$ is finite dimensional, $\PI(R)$ is assumed to be finitely generated, and $\mathsf{pi}(R)_{> \gamma}$ is is closed under finite direct sums.

Choosing a basis $(\phi_i)_{i \in I}$ of 
\[
\bigoplus_{{M \in \mathsf{pi}(R)_{> \gamma}}} \Hom_R(\Gamma,M)
\]
determines a morphism $ \Phi:\Gamma \rightarrow \Lambda $
as the direct sum $\Phi = \oplus_{i \in I} \phi_i$. One can check that $\Phi$ is a left $\mathsf{pi}(R)_{>\gamma}$ approximation.

The existence of a left approximation with a finite length target implies the existence of a minimal left approximation by, for example, \cite[Theorem I.2.4]{ARS}, which shows such a minimal approximation can be constructed from an approximation by projection onto a summand. Hence the existence of the approximation $\Phi:\Gamma \rightarrow \Lambda$ ensures that a minimal left $\mathsf{pi}(R)_{>\gamma}$ approximation $\alpha_{\Gamma}:\Gamma \rightarrow \Gamma_{>\gamma}$ exists.
\end{proof}

\begin{defn} \label{goodleftapproximation}
We say that $\PI(R)$ has \emph{good left approximations} if \[\Hom_R(\coker \alpha_\Gamma, \PI(R))=0 \] for all principal left $R$-ideals $\Gamma$.
\end{defn}

\begin{lem}\label{L:Goodleftapprox}
If $R$ is an ideally ordered monomial algebra, then for a principal ideal $\Gamma$ of layer $\gamma$ the minimal left $\mathsf{pi}(R)_{>\gamma}$ approximation is surjective. Hence when $R$ is ideally ordered $\PI(R)$ has \emph{good left approximations}.
\end{lem}
\begin{proof}
Since $R$ is ideally ordered, we can use Lemma \ref{L:PIareMonom} to replace any principal $R$-ideal by an isomorphic monomial ideal wherever needed. 
In particular, without loss of generality let $\Gamma=Rg$ (with $g \in eR$ a monomial) be a principal left $R$-ideal of layer $\gamma$.

A surjection from $\Gamma$ to a principal ideal exists, $\Gamma \rightarrow 0$ as 0 is a principal ideal. Using that $R$ is finite dimensional there is a surjection to a principal ideal $\Gamma_{>\gamma}$ which has maximal dimension among all principal ideals that admit surjections from $\Gamma$. The existence of the surjection implies that $\Gamma$ and $\Gamma_{>\gamma}$ have the same head. In particular, we can assume that $\Gamma_{>\gamma}=Rn$ for a monomial $n \in eR$. Using Lemma \ref{L:StandardEpi}, the assignment $g \mapsto n$ defines an $R$-linear surjection $\alpha_\Gamma \colon \Gamma \to \Gamma_{>\gamma}$. 

We now claim that $\alpha_{\Gamma}$ is an approximation. To prove this we consider a principal ideal $\Lambda$ and will show that the induced map $\Hom_R(\Gamma_{>\gamma}, \Lambda) \rightarrow \Hom_R(\Gamma,\Lambda)$ is a surjection. Take a morphism $\beta \in \Hom_R(\Gamma,\Lambda)$. We aim to show that $\beta$ factors through $\alpha_{\Gamma}$ and hence is the image of some morphism in $\Hom_R(\Gamma_{>\gamma}, \Lambda)$.

To see this, take the induced surjection $\beta \colon \Gamma \to \im \beta$ and, as the image of a principal ideal in a principal ideal, $\im \beta \cong Rm$ (with a monomial $m \in eR$) is a principal left $R$-ideal. Using the ideally ordered condition on $R$ there is a surjection in at least one direction between $\im \beta$ and $\Gamma_{> \gamma}$. As $\Gamma_{>\gamma}$ is a principal ideal of maximal dimension with a surjection from $\Gamma$, it follows that $\dim \Gamma_{>\gamma} \ge \dim \im \beta $ and hence there is a surjection $\sigma \colon \Gamma_{>\gamma} \rightarrow \im \beta$. Using Lemma \ref{L:StandardEpi}, we can assume that $\sigma$ is given by $n \mapsto m$. Hence, the composition $\pi:=\sigma \circ \alpha_\Gamma$ is a surjection defined by $g \mapsto m$. Now Lemma \ref{L:GenToGen} shows that the surjection $\beta \colon \Gamma \to \im \beta$ factors over $\pi$. In particular, $\beta$ factors over $\alpha_\Gamma$. So $\alpha_\Gamma$ is an approximation.  

Finally, we claim that this approximation is minimal. To see this consider an endomorphism $\phi: \Gamma_{>\gamma} \rightarrow \Gamma_{>\gamma}$ such that $\phi \circ \alpha_{\Gamma} \cong \alpha_{\Gamma}$. Then as $\alpha_{\Gamma}$ is a surjection it follows that $\phi$ is a surjection, and hence an isomorphism.

By construction, $\coker \alpha_\Gamma = 0$ for all $\Gamma$ so  $\PI(R)$ has good left approximations.
\end{proof}

We give examples showing that our results above apply beyond the class of ideally ordered monomial algebras. 

\begin{ex} \label{E:Goodleftapproximation} \hfill
\begin{itemize} 
\item[(a)] Consider the monomial algebra $R=kQ/I$, where

\begin{align*}
\begin{aligned}
\begin{tikzpicture} [bend angle=15, looseness=1]
\node (c) at (-1,0) {$Q:=$};
\node (C0) at (0,0)  {$2$};
\node (C1) at (2,0)  {$1$};
\draw [->,bend left] (C1) to node[below]  {\scriptsize{$x$}} (C0);
\draw [->,bend right] (C1) to node[above]  {\scriptsize{$y$}} (C0);
\draw [->, looseness=24, in=52, out=128,loop] (C1) to node[above] {$\scriptstyle{a}$} (C1);
\draw [->, looseness=24, in=-38, out=38,loop] (C1) to node[right] {$\scriptstyle{b}$} (C1);
\draw [->, looseness=24, in=-128, out=-52,loop] (C1) to node[below] {$\scriptstyle{c}$} (C1);
\end{tikzpicture}
\end{aligned}
&\qquad \text{ and } \begin{aligned}
I=(a,b,c)^2+(yb,xc).
\end{aligned}
\end{align*}
This is not ideally ordered since there are no surjections between $Rb$ and $Rc$, however $\PI(R)$ still has good left approximations. It is a short exercise to find the $5$ isomorphism classes of indecomposable principal ideals and calculate their minimal left approximations. All but one of these minimal approximations are surjective, and the one which is not surjective has cokernel $S_1$, the simple at vertex 1. There are no morphisms from $S_1$ to any principal ideal, and hence $\PI(R)$ has good left approximations.

\item[(b)] Let $n>0$ be an integer. Consider the non-monomial algebras $R_n=kQ/I_n$  where 
\begin{align*}
\begin{aligned}
\begin{tikzpicture} 
\node (c) at (-1,0) {$Q:=$};
\node (C1) at (0,0)  {$1$};
\node (C2) at (1.5,-0.8)  {$2$};
\node (C3) at (1.5,0.8)  {$3$};
\node (C4) at (3,0)  {$4$};
\draw [->] (C1) to node[above]  {\scriptsize{$a$}} (C3);
\draw [->] (C1) to node[above]  {\scriptsize{$b$}} (C2);
\draw [->] (C2) to node[above]  {\scriptsize{$d$}} (C4);
\draw [->] (C3) to node[above]  {\scriptsize{$c$}} (C4);
\draw [->, looseness=24, in=-38, out=38,loop] (C4) to node[right] {$\scriptstyle{x}$} (C4);
\end{tikzpicture}
\end{aligned}
&\qquad \text{ and } \begin{aligned}
I_n=(x^n, ca-db,xc,xd).
\end{aligned}
\end{align*}
Again, $\mathsf{PI}(R_n)$ has good left approximations; it is a short exercise to find the $n+3$ principal ideals and calculate that the minimal left approximation for each one is surjective.
\end{itemize}
\end{ex}

\begin{prop}\label{P:LeftStrong} The algebra $E_R^{\PI}=\End_R(\PI(R))$ is left strongly quasi-hereditary with respect to the ideal layer function $l$ if and only if $\PI(R)$ has good left approximations with respect to $l$. 
\end{prop}
\begin{proof} Assume $\PI(R)$ has good left approximations $\alpha_\Gamma \colon \Gamma \to \Gamma_{>\gamma}$. 
Using the condition on $\coker \alpha_\Gamma$ and applying $\Hom_R(-, \PI(R))$ yields a short exact sequence
\begin{align}\label{E:ProjRes}
0 \to P(\Gamma_{>\gamma}) \xrightarrow{\iota(\Gamma)} P(\Gamma) \to \Delta(\Gamma) \to 0,
\end{align}
where $\iota(\Gamma)=\Hom_R(\alpha_\Gamma, \PI(R))$ and $\Delta(\Gamma)$ denotes the cokernel of $\iota(\Gamma)$. 
We claim that the ideal layer function defines a left strongly quasi-hereditary structure on $E_R^{\PI}$ such that the $\Delta(\Gamma)$ are standard modules. To see this we have to show that \eqref{E:ProjRes} satisfies conditions (a) and (b) outlined in Definition \ref{D:RingelStrongQH}. Since all direct summands of $P(\Gamma_{>\gamma})$ are of the form $P(\Lambda)$ with $l(\Lambda) > \gamma$ condition (a) is satisfied by construction. Using the anti-equivalence $\Hom_R(-, \PI(R)) \colon \add \PI(R) \to \proj$-$E_R^{\PI}$ condition (b) translates to: every $R$-linear non-isomorphism $\nu\colon \Gamma \to \Lambda$ with $\Lambda \in \mathsf{pi}(R)_{\geq \gamma}$ factors over $\alpha_\Gamma$. By definition of $\alpha_\Gamma$ this holds for $\Lambda \in \mathsf{pi}(R)_{>\gamma}$. If $\Lambda \in \mathsf{pi}(R)_{=\gamma}$, then $\nu$ cannot be surjective for otherwise it is an isomorphism since $\dim_k \Lambda =\dim_k \Gamma$. Therefore, $\im \nu \subsetneq \Lambda$ is a principal left $R$-ideal with $l(\im \nu) > l(\Lambda)=\gamma$. So $\nu$ factors over $\alpha_\Gamma$. 

\smallskip

To see the converse direction,
assume that $\PI(R)$ does not have good left approximations. Then there exists a principal left $R$-ideal $\Gamma$  such that $\Hom_R(\coker \alpha_\Gamma, \PI(R)) \neq 0$. Assume that $E_R^{\PI}$ is quasi-hereditary with respect to the ideal layer function $l$ and let $\Delta(\Gamma)$ be the standard module corresponding to $\Gamma$. Since $\alpha_\Gamma$ is a minimal left $\mathsf{pi}(R)_{>\gamma}$ approximation
\begin{align*}\label{E:ProjRes2}
 P(\Gamma_{>\gamma}) \xrightarrow{\Hom_R(\alpha_\Gamma, \PI(R))} P(\Gamma) \to \Delta(\Gamma) \to 0,
\end{align*}
is the start of a minimal projective resolution of $\Delta(\Gamma)$. By our choice of $\Gamma$ the morphism $\Hom_R(\alpha_\Gamma, \PI(R))$ is not injective. Hence $\Delta(\Gamma)$ has projective dimension greater than $1$ and, using Definition \ref{D:Stronglyqh}, $A$ is not left strongly quasi-hereditary with respect to $l$ in this case.
\end{proof}

\begin{rem} \label{R:Standardinclusions}
 Assume that $E_R^{\PI}$ is quasi-hereditary with respect to the ideal layer function.
One can show that as a set the standard module $\Delta(\Gamma)$ is given by all (residue classes of) monomorphisms starting in $\Gamma$. Indeed if $\nu \colon \Gamma \to \Lambda$ is not a monomorphism then an argument along the lines of the proof of the proposition shows that $\nu$ factors over $\alpha_\Gamma$ and therefore corresponds to the zero element in $\Delta(\Gamma)$.
\end{rem}

Proposition \ref{P:LeftStrong} is related to
\cite[Theorem 5]{Iyamasfiniteness} by Ringel. He shows that for an $R$-module $M$ there exists an $R$-module $N$ such that $\End_R(M \oplus N)$ is left strongly quasi-hereditary and all the indecomposable summands $N$ are submodules of $M$. In particular, if $M$ is an $R$-module such that all submodules are isomorphic to direct summands of $M$, then $\End_R(M)$ is left strongly quasi-hereditary. We will see in Theorem \ref{T:Main} that $\PI(R)$ has this property if $R$ is ideally ordered monomial. However, our proof of Theorem \ref{T:Main} uses Proposition \ref{P:LeftStrong}, so we cannot apply Ringel's result in our approach.

Now we look at the `dual' side.
First we `dualise' Definition \ref{goodleftapproximation} using the same notation.
\begin{defn} \label{goodrightapproximation}
For every principal left ideal $\Gamma$ there is a minimal right $\mathsf{pi}(R)_{>\gamma}$ approximation  
\[
\rho_\Gamma \colon \Gamma_{>\gamma} \to \Gamma
\]
with $\Gamma_{>\gamma} \in \mathsf{pi}(R)_{>\gamma}$. We say that $\PI(R)$ has \emph{good right approximations} if 
\[\Hom_R(\PI(R), \ker \rho_\Gamma)=0.\]  Since $\PI(R)$ contains $R$ as a direct summand this is equivalent to 
\[\ker \rho_\Gamma=0 \] for all principal left $R$-ideals $\Gamma$. 
\end{defn}

\begin{ex} \label{Example: Good right approximations}
\begin{itemize}
\item[(a)] Let $R$ be a finite dimensional monomial algebra. Then $\PI(R)$ has good right approximations. Indeed, let $\Gamma$ be a principal left $R$ ideal. Since $R$ is monomial, $\rad \Gamma$ is a direct sum of principal left ideals in $\mathsf{pi}(R)_{>\gamma}$ and the natural inclusion $\rad \Gamma \to \Gamma$ gives the desired minimal right approximation $\rho_\Gamma$. 
\item[(b)] The algebra in Example \ref{E:Goodleftapproximation} (b) does not have good right approximations: the minimal right approximation of the projective module $P_1$ is $P_2 \oplus P_3 \rightarrow P_1$ and this has kernel $S_4$.
\end{itemize}

\end{ex}

The following result is proved dually to Proposition \ref{P:LeftStrong} 

\begin{prop} \label{P:RightStrong} $E_R^{\PI}=\End_R(\PI(R))$ is right strongly quasi-hereditary with respect to the ideal layer function $l$ if and only if $\PI(R)$ has good right approximations. For example, this holds if $R$ is finite dimensional monomial.
\end{prop}

Combining Proposition \ref{P:LeftStrong} and \ref{P:RightStrong} with Lemma \ref{L:Goodleftapprox} and Example \ref{Example: Good right approximations}(a) yields the following theorem.
\begin{thm} \label{thm: ideally ordered implies strongly quasi-hereditary}
If $R$ is an ideally ordered monomial algebra, then $E_R^\PI$ is both left and right strongly quasi-hereditary with respect to the ordering induced by the ideal layer function.
\end{thm}

We let ${\mathcal Filt}(\Delta)$ and ${\mathcal Filt}(\nabla)$ denote the full subcategories of $E_R^{\PI}$-$\textnormal{mod}$ of objects filtered by standard and costandard modules respectively.

\begin{rem}\label{R:CSFac}
 Assume that $E_R^{\PI}$ is quasi-hereditary with respect to the ideal layer function. Similarly to the case above, one can show that as a set a costandard module $\nabla(\Lambda)$ is given by all surjections ending in $\Lambda$. In particular, each costandard module has head $S(\Pi)$ for some indecomposable projective $R$-module $\Pi$ and ${\mathcal Filt}(\nabla) \subseteq \fac(P(R))$.
\end{rem}

\begin{cor} \label{C:ClosedSubFac} If $\PI(R)$ has good  right and left approximations, then ${\mathcal Filt}(\Delta)$ is closed under submodules and ${\mathcal Filt}(\nabla)$ is closed under quotients.
\end{cor}
\begin{proof}
If $\PI(R)$ has good left approximations, then $E_R^{\PI}$ is left strongly quasi-hereditary by Proposition \ref{P:LeftStrong}, and hence all standard objects have projective dimension 1. By \cite[Proposition A.1]{Iyamasfiniteness} all standard modules having projective dimension 1 is equivalent to ${\mathcal Filt}(\nabla)$ being closed under quotients.

The analogous dual statement, using Proposition \ref{P:RightStrong}, shows that when $\PI(R)$ has good right approximations then ${\mathcal Filt}(\Delta)$ is closed under submodules.
\end{proof}

\section{The characteristic tilting module and Ringel duality} \label{S:chartilting}

In the following section we first recall the characteristic tilting module $T$ associated to a quasi-hereditary algebra. Then we show that our algebras $E_R^{\PI}$ are ultra strongly quasi-hereditary in the sense of Conde \cite{Conde17} and use this to determine a subcategory of the additive hull $\add(T)$ of $T$ (Corollary \ref{C:Subcat}). In the proof of our main Theorem \ref{T:Main} we show that these categories coincide for ideally ordered monomial algebras $R$ and as a consequence establish our Ringel-duality formula in this setup.

The following proposition can be found in Ringel \cite{Ringel}, which is based on work of Auslander \& Reiten \cite{AuslanderReitenContravariantlyFinite} and Auslander \& Buchweitz \cite{AuslanderBuchweitz}.

\begin{prop}\label{P:char}
Let $A$ be a quasi-hereditary algebra. Then there exists a tilting module $T \in A$-$\mod$ such that 
\begin{align*}
\add(T)={\mathcal Filt}(\Delta) \cap {\mathcal Filt}(\nabla),
\end{align*}
where ${\mathcal Filt}(\Delta) \cap {\mathcal Filt}(\nabla)$ is the full subcategory of $A$-modules with filtrations by both standard and costandard modules.
\end{prop}

\begin{defn}
A tilting module $T$ occurring in Proposition \ref{P:char} is called a \emph{characteristic tilting module}. The \emph{Ringel dual} $\mathfrak{R}(A)$ of an algebra $A$ is defined by
\[
\mathfrak{R}(A):=\End_A(T)^{\op}
\]
for $T$ the basic characteristic tilting module consisting of one copy of each indecomposable module in ${\mathcal Filt}(\Delta) \cap {\mathcal Filt}(\nabla)$ up to isomorphism: i.e. we assume $\mathfrak{R}(A)$ is a basic algebra.
\end{defn}

The notion of a ultra strongly quasi-hereditary algebras was introduced by Conde, see  \cite[Section 2.2.2]{Conde17}. 

\begin{defn}
A quasi-hereditary algebra $A$ is \emph{left ultra strongly quasi-hereditary} if a projective module $P_i$ is filtered by costandard modules whenever the corresponding costandard module $\nabla_i$ is simple.
\end{defn}

Let $e_0 \in E_R^{\PI}=\End_R(\PI(R))$ be the idempotent corresponding to the direct summand $R$ of $\PI(R)$. Note that $e_0$ is primitive if and only if $R$ is local. We have the following. 

\begin{prop} \label{P:AefiltCS}
Let $R$ be a finite dimensional algebra.
Assume that $\PI(R)$ has good left approximations, so that $E_R^{\PI}$ is left strongly quasi-hereditary with respect to the ideal layer function $l$. Then the following conditions are equivalent:
 \begin{itemize}
 \item[(a)] $E_R^{\PI}e_0$ is filtered by costandard objects.
 \item[(b)] $\alpha_{\Gamma}\colon \Gamma \to \Gamma_{>\gamma}$ is surjective for all principal $R$-ideals $\Gamma$.
 \item[(c)] $E_R^{\PI}$ is left ultra strongly quasi-hereditary.
 \end{itemize}
If $R$ is monomial then these conditions are equivalent to  
\begin{itemize}
 \item[(d)] $R$ is ideally ordered.
 \end{itemize}
\end{prop}

\begin{proof}
We first show that (b) implies (a). By \cite[Theorem 4]{Ringel}, it suffices to show that $\Ext^1_{E_R^{\PI}}(\Delta(\Gamma), P(Re_i))=0$ for all principal left $R$ ideals $\Gamma$ and all primitive idempotents $e_i \in R$. We can assume that $\Delta(\Gamma)$ is not projective. Then applying $\Hom_{R}(-,\PI(R))$ to $\alpha_{\Gamma}$ produces the projective resolution 
\begin{align*}
0 \to P(\Gamma_{>\gamma}) \xrightarrow{\iota(\Gamma)} P(\Gamma) \to \Delta(\Gamma) \to 0,
\end{align*}
and we have to show that every morphism $P(\Gamma_{>\gamma}) \to P(Re_i)$ factors over $\iota(\Gamma)$. Applying the anti-equivalence given in equation (\ref{E:AntiEquivalence})
translates this statement to: every morphism $\varphi\colon Re_i \to \Gamma_{>\gamma}$ factors over $\alpha_\Gamma \colon \Gamma \to \Gamma_{>\gamma}$. This holds since $Re_i$ is projective and $\alpha_\Gamma$ is surjective by assumption.

Conversely, if $\alpha_\Gamma$ is not surjective for some principal ideal $\Gamma$ then there exists $x \in \Gamma_{>\gamma} \setminus \im \alpha_\Gamma$. Since $R$ is free there is an $R$-linear map $R \to \Gamma_{>\gamma}$, $1 \mapsto x$, which by construction does not factor over $\alpha_\Gamma$. In combination with the anti-equivalence and projective resolution above this shows $\Ext^1_{E_R^{\PI}}(\Delta(\Gamma), P(R)) \neq 0$ and \cite[Theorem 4]{Ringel} completes the proof that $(a)$ implies $(b)$.

That (a) is equivalent to (c) follows from the fact that $\nabla(\Lambda)$ is simple if and only if $\Lambda$ is projective, see Remark \ref{R:CSFac}, and hence $\nabla(\Lambda)$ simple implies $P(\Lambda)$ is a direct summand of $E^{\PI}_R e_0$.

Let $R$ be monomial. The implication (d) $\Rightarrow$ (b) follows from Lemma \ref{L:Goodleftapprox}. We now assume (b) and prove the converse. 

Firstly, for any indecomposable principal ideal $\Gamma$ the minimal left approximation $\alpha_\Gamma \colon \Gamma \to \Gamma_{>\gamma} $ is surjective by assumption (b), and we claim that $\Gamma_{>\gamma}$ is indecomposable.

To show this take $p \in eR$ for $e$ a primitive idempotent and consider the principal ideal $\Gamma \cong Rp$. Now suppose that there is a decomposition $\Gamma_{>\gamma} \cong \bigoplus Rq_i$ for some principal ideals $Rq_i$. As $\alpha_{\Gamma}$ is surjective, after relabelling  we can assume that the image of $p$ is $(q_1, \dots, q_n)$ and $q_1 \neq 0$. As the morphism $\alpha_{\Gamma}$ is surjective there must exist some $r \in Re$ such that $\alpha_{\Gamma}(rp)=(q_1,0, \dots, 0)$; i.e. $rq_1 =q_1$ and $r q_j = 0$ for $j \ge 2$. As $R$ is monomial, by considering the monomial of lowest degree occurring in $q_1$ and $rq_1 =q_1$ we can see that the degree 0 primitive idempotent $e$ must occur in $r$. Then we can rewrite $r=e+r'$ where all monomials occurring in $r'$ have degree greater than $0$. As a result, $q_j$ must be zero as $0=rq_j= q_j + r'q_j$ so there can be no non-zero monomial of lowest degree occurring in $q_j$. Hence $q_j=0$ for $j\ge 2$, the decomposition is a trivial decomposition $Rq_1 \cong Rq_1 \oplus 0 \oplus \dots \oplus 0$, and $\Gamma_{> \gamma}$ is indecomposable.

This allows the successive construction of left $\mathsf{pi}_{>k}$ approximations starting with the indecomposable principal ideal $Re$ 
\[
Re \xrightarrow{\alpha_{Re}} Re_{>i_1} \xrightarrow{\alpha_{{Re}_{>i_1}}} Re_{>i_2} \xrightarrow{\alpha_{Re_{>i_2}}} \dots \xrightarrow{\alpha_{Re_{>i_{n-1}}}} Re_{>i_n}
\]
where $i_1=l(Re)$, $i_{j+1}=l(Re_{>i_j})$, and $\alpha_{Re_{>i_j}}:Re_{>i_j} \rightarrow Re_{>i_{j+1}}$ is the minimal left $\mathsf{pi}(R)_{> i_{j+1}}$ approximation. Each $Re_{>i_j}$ is indecomposable, and the composition $\alpha_k: Re \rightarrow Re_{>i_k}$ of the left approximations is again a left approximation. 

We claim that any indecomposable principal ideal $Rx$ with $x \in eR$ is isomorphic to one of these successive approximations. To see this choose $k$ to be maximal such that $l(Rx) > i_k$. Then there is a surjection $\pi:Re \rightarrow Rx$, and as $Rx \in \mathsf{pi}(R)_{>i_k}$ this must factor through the left approximation $\alpha_k:Re \rightarrow Re_{>i_k}$ by a surjection $\phi:Re_{>i_k} \rightarrow Rx$. In particular, $\dim Re_{>i_k} \ge \dim Rx$ so $l(Re_{>i_k}) \le l(Rx)$. But, by the definition of $k$, it is true that $i_{k+1}=l(Re_{>i_k})  \ge l(Rx)$, hence it must be the case that $l(Re_{>i_k}) = l(Rx)$ so $\dim \, Re_{>i_k} =\dim \, Rx$ and hence the surjective morphism $\phi$ is an isomorphism $Re_{>i_k}  \cong  Rx$.

Finally, any pair $Rx$ and $Ry$ of principal ideals with $x, y \in eR$ occur (up to isomorphism) in the successive approximation sequence, in which every morphism is surjective by assumption (b), and hence there is a surjection between them. This proves that the ideally ordered condition holds.
\end{proof}

\begin{ex}
The non-monomial algebra in Example \ref{E:Goodleftapproximation} (b) satisfies the equivalent conditions (a), (b) and (c) of the theorem.
\end{ex}

\begin{cor}\label{C:Subcat} Suppose that $\PI(R)$ has both good left and right approximations. Then
$\sub(E_R^{\PI}e_0) \cap \fac(E_R^{\PI}e_0) \subseteq {\mathcal Filt}(\Delta) \cap {\mathcal Filt}(\nabla) = \add(T)$.
\end{cor}
\begin{proof}
By the definition of a quasi-hereditary algebra every projective module is filtered by standard modules. Therefore, $(E_R^{\PI}e_0)^{\oplus n} \in {\mathcal Filt}(\Delta)$ and by Proposition \ref{P:AefiltCS}(a), we also have $(E_R^{\PI}e_0)^{\oplus n} \in {\mathcal Filt}(\nabla)$. Now Corollary \ref{C:ClosedSubFac} yields $\sub(E_R^{\PI}e_0) \subseteq {\mathcal Filt}(\Delta)$ and  $\fac(E_R^{\PI}e_0) \subseteq {\mathcal Filt}(\nabla)$. This implies the claim.
\end{proof}

\begin{rem}
In combination with Remark \ref{R:CSFac}, we see that when $\PI(R)$ has both good left and right approximations $\fac(E_R^{\PI}e_0)={\mathcal Filt}(\nabla)$. For ideally ordered monomial algebras $R$, Theorem \ref{T:Main}(e) shows that $\sub(E_R^{\PI}e_0)={\mathcal Filt}(\Delta)$ holds as well.
\end{rem}

\begin{rem}
Let $R=R_2$ be the non-monomial algebra from Example \ref{E:Goodleftapproximation} (b). The algebra $E_R^{\PI}=\End_R(\PI(R))$ is left ultra strongly quasi-hereditary with respect to the ideal layer function (in particular, $E_R^{\PI}e_0$ is filtered by costandard modules) but not right strongly quasi-hereditary, so ${\mathcal Filt}(\Delta)$ is not closed under subobjects. It turns out that there is precisely one indecomposable subobject of $E_R^{\PI} e_0$ which is not filtered by standard modules. This module is also a quotient of $E_R^{\PI} e_0$ and therefore
$\sub(E_R^{\PI}e_0) \cap \fac(E_R^{\PI}e_0) \not\subseteq {\mathcal Filt}(\Delta) \cap {\mathcal Filt}(\nabla)=\add(T)$. Restricting to the local submodules of $E_R^{\PI}e_0$ yields the desired inclusion into $\add(T)$ in this case and can be used to show a version of the Ringel duality formula \eqref{E:AlternativeFormula} in this example. Unfortunately, we don't know how to fit this example into a larger framework.
\end{rem}

\section{An equivalence from idempotents}
\noindent
In this section, we show that there is an equivalence of categories 
\[
\Hom_A(Ae_0, -)\colon \sub(Ae_0) \cap \fac(Ae_0) \to \sub(e_0Ae_0).
\]
where $A=E^\PI_R$ for a finite dimensional algebra $R$ with $\PI(R)$ finitely generated and $e_0 \in A$ is the idempotent corresponding to the projection onto $R$.

To show this we recall several well-known lemmas. 

\begin{lem}\label{L:Serre}
Let $\ca$ be an abelian category with Serre subcategory $\cs$ and let $q \colon \ca \to \ca/\cs$ be the quotient functor. Then the restriction of $q$
\begin{align*}
^\perp\cs \cap \cs^\perp \to \ca/\cs  
\end{align*}
is fully faithful. Here 
\begin{align*}
&^\perp\cs:=\{ X \in \ca \mid \Hom_\ca(X, S)=0 \text{  for all } S \text{ in } \cs \}\text{, and } \\
&\cs^\perp:=\{ Y \in \ca \mid \Hom_\ca(S, Y)=0 \text{  for all } S \text{ in } \cs \}.
\end{align*}
\end{lem}
\begin{proof}
This follows from the description of homomorphism spaces in the quotient category as colimits. Indeed for $X, Y \in {^\perp\cs \cap \cs^\perp}$ the colimit describing $\Hom_{\ca/\cs}(X, Y)$ is taken over the single pair of subobjects $(X, 0)$ and the quotient functor sends a morphism $f\colon X \to Y$ to $f$.
\end{proof}

The following lemma can be found in \cite[Proposition 5.3 (b)]{GeigleLenzingPerpendicularCategories}

\begin{lem}\label{L:MultByIdemp}
Let $B$ be a noetherian ring and let $e \in B$ be an idempotent. Then 
\begin{align*}
F=\Hom_B(Be, -)\colon B-\mod \to eBe-\mod 
\end{align*} 
is an exact quotient functor with kernel $B/BeB-\mod$. In particular, 
$B/BeB-\mod$ is a Serre-subcategory in $B-\mod$.  
\end{lem}

\begin{cor} \label{C:perp1} In the notation of Lemma \ref{L:MultByIdemp}, we have $\fac(Be) \subseteq ^\perp\!(B/BeB-\mod)$.
\end{cor}
\begin{proof}
Consider $N \in \fac(Be)$ and $M \in B/BeB$-$\mod$. Applying the right exact functor $\Hom_B(-,M)$ to the surjection $Be \rightarrow N \rightarrow 0$ yields the injection $0 \rightarrow \Hom_B(N,M) \rightarrow \Hom_B(Be,M)$. As $B/BeB$-$\mod$ is the kernel of $\Hom_B(Be,-)$ and $M \in B/BeB$-$\mod$ it follows that $\Hom_B(Be,M)=0$ and hence $\Hom_B(N,M)=0$. 
\end{proof}

From now on let $A=E^{\PI}_R$ for some finite dimensional algebra $R$, such that $\PI(R)$ is finitely generated.

\begin{lem} \label{L:Soc}
 In the notation of Section \ref{S:chartilting}, we have $\mathsf{soc} \, Ae_0 \subseteq S_0^{\oplus n}$ for some natural number $n$. Here $S_0=Ae_0/\rad Ae_0$ is the semi-simple head of $Ae_0$.
\end{lem}
\begin{proof}
Indeed $Ae_0$ consists of all $R$-homomorphisms $R \to \PI(R)$. Let $\Lambda$ be a principal left $R$-ideal. If $R \to \Lambda$ is non-zero, then the composition with the canonical inclusion $R \to \Lambda \to R$ is non-zero. Therefore every maximal sequence of non-zero morphisms starting in $R$ ends in $R$, proving the claim.
\end{proof}

\begin{cor} \label{C:perp}
$\sub(Ae_0) \subseteq (A/Ae_0A - \mod)^\perp$
\end{cor}
\begin{proof}
Assume that $f \colon X \to U$ is a non-zero map, where $U$ in $\sub(Ae_0)$ and $X$ in $A/Ae_0A - \mod$.  
Lemma \ref{L:Soc} implies that $\im f$ contains a non-zero direct summand of $S_0$. But $\im f \in A/Ae_0A - \mod$ since $X$ is contained in $A/Ae_0A - \mod$. It follows that $\im f$ has no submodule which is a direct summand of $S_0$. Contradiction. So there is no non-zero morphism $f\colon X \to U$.  
\end{proof}

The following statement is the main result of this section.

\begin{prop}\label{P:Key}
The exact functor $F=\Hom_A(Ae_0, -)$ restricts to an additive equivalence
\begin{align}
\sub(Ae_0) \cap \fac(Ae_0) \to \sub(e_0Ae_0) \cap \fac(e_0Ae_0) =\sub(e_0Ae_0). 
\end{align}
\end{prop}
\begin{proof}
The equality on the right follows from the fact that $\fac(e_0Ae_0)=e_0Ae_0-\mod$.
Since $F$ is exact and maps an $A$-module $M$ to $e_0M$, the restriction is well-defined. We can apply Lemma \ref{L:Serre} to $q=F$ to deduce that $F$ is fully faithful. Indeed, by Lemma \ref{L:MultByIdemp} $F$ is a quotient functor corresponding to the Serre subcategory $A/Ae_0A-\mod$ and Corollaries \ref{C:perp1} and Corollary \ref{C:perp} show that the required orthogonality conditions are satisfied.

It remains to show that $F$ is essentially surjective. Let $U \subseteq (e_0Ae_0)^{\oplus n}$ be generated by $u_1, \ldots, u_n \in (e_0Ae_0)^n$. The $u_i$ are elements of $(Ae_0)^n$. Let $V \subseteq (Ae_0)^{\oplus n}$ be the $A$-submodule generated by the $u_i$. One can check that $F(V)=U$ and since $e_0 u_i=u_i$ for all $i$ 
$V$ is a factor module of $(Ae_0)^{\oplus m}$ for some $m$. This shows that $V$ is contained in $\sub(Ae_0) \cap \fac(Ae_0)$ and completes the proof.

\end{proof}

\section{Proof of Ringel duality formula} \label{Section:Proof of main theorem}

In this section we prove the following main result of this paper, which is an extended version of Theorem \ref{T:PrelimMain} stated in the introduction.

\begin{thm}\label{T:Main} Let $R$ be a finite dimensional ideally ordered monomial algebra and $E_R=\End_R(\SUB(R))$. Then $E_R$ is quasi-hereditary and the Ringel duality formula 
\[
\mathfrak{R}(E_R) \cong (E_{R^{\op}})^{\op}.
\]
holds. More explicitly, where $\dagger$ denotes the standard $k$-duality,
\begin{align} \label{E:RingelDualityFormula}
\mathfrak{R}(\End_R(\SUB(R))) \cong \End_R(\FAC(R_R^{\dagger})) \cong \End_{R^{\op}}(\SUB(R^{\op}))^{\op},
\end{align} and if we consider $\sub(R):=\add \SUB(R)$ and $\fac(R^{\dagger}):=\add \FAC(R_R^{\dagger})$ as exact categories with split exact structures then this Ringel duality induces the derived equivalence
\[
D^b(\sub(R)) \cong D^b(\fac(R^{\dagger})).
\]
Moreover:
\begin{itemize}
\item[(a)] Every indecomposable submodule of $R^n$ is isomorphic to a principal left ideal, every principal left ideal is isomorphic to a monomial ideal, and hence $\sub(R) \cong \mathsf{pi}(R)$ so $E_R^{\PI} \cong E_R$. 

\item[(b)] The algebra $E_R$ is left and right strongly quasi-hereditary with respect to the ideal layer function. In particular, $E_R$ has global dimension at most $2$. Moreover, it is left ultra strongly quasi-hereditary in the sense of Conde \cite{Conde17}.

\item[(c)] The ideal order is the unique order defining a quasi-hereditary structure on $E_R$ if $R$ is local and satisfies the following condition: if there exists a surjection $\Lambda \to \Gamma$ between principal left ideals, then there is an inclusion $\Gamma \to \Lambda$.
\item[(d)] Let $T$ be the characteristic tilting module of $E_R$ and $e_0 \in E_R$ be the idempotent corresponding to $R$. Then there is an equality of subcategories
$\add T = \sub(E_R e_0) \cap \fac(E_R e_0).$
In other words, the indecomposable direct summands $T_i$ of $T$ are precisely those indecomposable $E_R$-modules which are both quotients and submodules of the projective module $E_R e_0$.

\item[(e)] We can describe the subcategories ${\mathcal Filt}(\Delta)$ and ${\mathcal Filt}(\nabla)$ of $E_R$-$\mod$ as follows: 
\begin{align}
&{\mathcal Filt}(\Delta) = \sub(T) = \sub(E_R e_0) = \sub(E_R), \text{ and } \label{E:filtDelta}\\
&{\mathcal Filt}(\nabla) = \fac(T) \, = \fac(E_R e_0). \label{E:filtNabla}
\end{align}
 \end{itemize}
\end{thm}

\begin{proof}

We first prove the main Ringel duality formula, and in the process also prove (a) and (d). Let $E_R^{\PI}=\End_R(\PI(R))$ and let $e_0 \in E_R^{\PI}$ be the idempotent corresponding to $R$. By Corollary \ref{C:Subcat}, we have an inclusion
\begin{align} \label{E:inclusion}
\sub(E_R^{\PI}e_0) \cap \fac(E_R^{\PI}e_0) \subseteq \add(T)
\end{align}
where $T$ is the characteristic tilting module for $E_R^{\PI}$. In combination with Proposition \ref{P:Key}, we get an inclusion \begin{align} \label{E:Incl}\sub(R^{\op}) \to \add(T), \end{align} since $e_0E_R^{\PI}e_0 \cong \End_R(R) \cong R^{\op}$.  Let $p$ (respectively, $p^{\op}$) be the number of indecomposable direct summands of $\PI(R)$ (respectively, $\PI(R^{\op})$) 
By definition of $E_R^{\PI}$, the number $p$ also equals the number of simple $E_R^{\PI}$-modules.
Which in turn equals the number of indecomposable summands of $T$ since $T$ is tilting. Let $s$ (respectively, $s^{\op}$) be the number of indecomposable direct summands of $\SUB(R)$ (respectively, $\SUB(R^{\op})$). By \eqref{E:Incl}, $s^{\op} \leq p$ (in particular, $s^{\op}$ is finite). Moreover, $\PI(R) \subseteq \SUB(R)$ implies $p \leq s$. It follows from \cite[Theorem 1.1]{RingelRepDim} that $s=s^{\op}$. Summing up, we have that $s^{\op} = p = s$. In particular, this yields equivalences
$\mathsf{pi}(R) \cong \sub(R)$, and therefore $E_R \cong E_R^{\PI}$ so proves (a). Moreover, using $s^{\op} = p$ and Proposition \ref{P:Key} the inclusions \eqref{E:inclusion} and \eqref{E:Incl} are equivalences
\begin{align} \label{E:equivalences}\sub(R^{\op}) \cong \add T = \sub\left(E_R^{\PI}e_0\right) \cap \fac\left(E_R^{\PI}e_0\right).\end{align}
In particular, this shows part (d).

By definition, the Ringel dual of $E_R$ is $\mathfrak{R}(E_R)=\End_{E_R}(T)^{\op}$. Using $\sub(R^{\op}) \cong \add T$ we obtain $\End_{E_R}(T)^{\op} \cong \End_{R^{\op}}(\SUB(R^{\op}))^{\op}$. Under the standard $k$-duality the latter identifies with $\End_R(\FAC(R_R^{\dagger}))$. This completes the proof of the main Ringel duality statement as given in formula \ref{E:RingelDualityFormula}. As a consequence we get the equivalence $D^b(\sub(R)) \cong D^b(\fac(R^{\dagger}))$.

We now consider part (b). By part (a) we know $E_R \cong E_R^{\PI}$, and as $R$ is ideally ordered Theorem \ref{thm: ideally ordered implies strongly quasi-hereditary} implies that $E_R^{\PI}$ is both left and right strongly quasi-hereditary with respect to the ideal layer function. An algebra which is left and right strongly quasi-hereditary with respect to the same ideal layer function has global dimension at most two by \cite[first Proposition in A.2]{Iyamasfiniteness}. Proposition \ref{P:AefiltCS} shows that $E_R \cong E_R^{\PI}$ is also left ultra strongly quasi-hereditary, and so completes the proof of statement (b).

We now prove (c).  Let $[M:S]$ denote the number of simple $E_R$-modules $S$ that occur in a Jordan Holder filtration of an $E_R$-module $M$. If a partial ordering on $I$ induces a quasi-hereditary structure, then $[\Delta_i,S_i]=1$ for all $i \in I$; as $k$ is algebraically closed this is equivalent to $\End_{E_R}(\Delta_i) \cong k$, see \cite[Lemma 1.6]{DlabRingel92TsukubaConferenceProc}.

Using the additional assumption in (c) that $R$ is local, the ideally ordered condition produces a surjection between any two summands of $\PI(R)$ (as all principal ideals are monomial by Lemma \ref{L:PIareMonom}). Hence the ideal layer function induces an ordering on the summands of $\PI(R)$ of the form $\Lambda_0 < \Lambda_1 < \cdots < \Lambda_t$. Now consider another partial order that also produces a quasi-hereditary ordering.

We first prove that both orderings have the same maximal element. If $\Lambda_i$ is maximal with respect to the new order, then the projective module $P_i:=P(\Lambda_i)$ is also a standard module in this order. If the new order gives rise to a quasi-hereditary structure then, as $P_i$ is standard in this ordering, $[P_i:S_i]=1$. As $P_i$ is projective $[P_i,S_i] = \dim \Hom_{E_R}(P_i,P_i)$.  Under the anti-equivalence $\Hom_R(-,\PI(R))$, described in formula (\ref{E:AntiEquivalence}), this implies $\dim \End_{R}(\Lambda_i) =1$. Hence the identity morphism must equal socle projection so $\Lambda_i$ is the simple $R$-module, which is unique as $R$ is assumed to be local. The simple $R$-module is the largest summand $\Lambda_t$ of $\PI(R)$ under the ideal layer function ordering, and hence $i=t$.

Secondly, we assume that the orderings match for $k, k+1, \dots, t$, let $\Lambda_j<\Lambda_k$ be an immediate predecessor of $\Lambda_k$ under the new order, and aim to show that $j=k-1$. As $R$ is ideally ordered there is a surjection between $\Lambda_{j}$ and $\Lambda_{j+1}$ (where $\Lambda_{j+1}$ exists as $j<k\le t$). As they are labelled by the ideal layer function $\dim \Lambda_j > \dim \Lambda_{j+1}$ and there is a surjection $\Lambda_{j} \rightarrow \Lambda_{j+1}$. By the condition assumed in c), the existence of this surjection implies an inclusion $\Lambda_{j+1} \rightarrow \Lambda_{j}$. Together these produce a non-trivial endomorphism $\Lambda_j \to \Lambda_{j+1} \to \Lambda_j$ which does not factor over $\Lambda_i$ for $i>j+1$. Using the anti equivalence $\Hom_R(-,\PI(R))$ again, this translates into a non trivial endomorphism of $P_j$ that does not factor over $P_i$ for $i>j+1$. In particular, the standard object under the new order $\Delta_j$ is the cokernel of a morphism $P \rightarrow P_j$ where the summands of $P$ are projective modules $P_i$ such that $i>j$ under the new ordering, see \cite[Lemma 1.1$'$]{DlabRingel92TsukubaConferenceProc}. If $k \neq j+1$, then both the trivial endomorphism and the non-trivial endomorphism constructed above do not factor via $P$ and hence $\dim \Hom_{E_R}(P_j,\Delta_j) \ge 2$. By considering the images of these morphisms we see $[\Delta_j : S_j] \ge 2$. This would imply that the new ordering does not give a quasi-hereditary structure. Therefore $j=k-1$. 

Finally, by proceeding in this way we recover the ideal order and conclude that there is only one quasi-hereditary structure. 

We show part (e). To prove \eqref{E:filtDelta}, we explain the following chain of subcategories
\[
{\mathcal Filt}(\Delta) = \sub(T) \subseteq \sub(E_R e_0) \subseteq \sub(E_R) \subseteq {\mathcal Filt}(\Delta). 
\]
By part (b), $E_R$ is right strongly quasi-hereditary. The first equality holds for all right strongly quasi-hereditary algebras, for example by a dual version of Proposition A.1 in \cite{Iyamasfiniteness}. Using \eqref{E:equivalences} and part (a), we see that $T \in \sub(E_R e_0)$ so $\sub(T) \subseteq \sub(E_R e_0)$. The next inclusion follows from $E_R e_0 \subseteq E_R$. The last inclusion holds for any right strongly quasi-hereditary algebra using that $E_R \in {\mathcal Filt}(\Delta)$, which is closed under submodules as noted in Corollary \ref{C:ClosedSubFac}. Using \eqref{E:equivalences} and the fact that $E_R$ is left ultra strongly quasi-hereditary by part (b), dual arguments establish the following chain
\[
{\mathcal Filt}(\nabla) = \fac(T) \subseteq \fac(E_R e_0) \subseteq {\mathcal Filt}(\nabla).
\]
(the last inclusion was also shown in the proof of Corollary \ref{C:Subcat}).
This implies \eqref{E:filtNabla} and completes the proof of part (e).
\end{proof}

For a monomial algebra $R$ there is an equivalence of additive categories $\langle \rad^i R \, | \,  i=1, \dots m \text{ for $\rad^mR \cong 0$} \rangle \cong \mathsf{pi}(R)$, and so $E_R^{\PI}$ is Morita equivalent to $\End_R(\bigoplus_{l=1}^m \rad^l R)$. This construction is considered in the general context of pre-radicals in Conde's thesis. An additional special feature of the ideally ordered algebras is that $\mathsf{pi}(R) \cong \sub(R)$, and this property does not hold for general monomial algebras. For example, consider the following example that was communicated to us by Xiao-Wu Chen.

\begin{ex} \label{ex:Not all Sub are PI}
Let $R$ be the path algebra of the following quiver with monomial relations.
\begin{align*}
\begin{aligned}
\begin{tikzpicture} []
\node (C1) at (0,0)  {$1$};
\node (C2) at (1.5,0)  {$2$};
\node (C3) at (3,0)  {$3$};
\draw [->, bend left = 20] (C1) to node[gap]  {\scriptsize{$x_1$}} (C2);
\draw [->] (C1) to node[gap]  {\scriptsize{$x_2$}} (C2);
\draw [->, bend right = 20] (C1) to node[gap]  {\scriptsize{$x_3$}} (C2);
\draw [->, bend left = 20] (C2) to node[gap]  {\scriptsize{$y_1$}} (C3);
\draw [->] (C2) to node[gap]  {\scriptsize{$y_2$}} (C3);
\draw [->, bend right = 20] (C2) to node[gap]  {\scriptsize{$y_3$}} (C3);
\end{tikzpicture}
\end{aligned}
\begin{aligned}
y_jx_i = 0 \text{ for $i \neq j$.}
\end{aligned}
\end{align*}
Then the left ideal $I=(x_1+x_2, \, x_2+x_3)$ is indecomposable but not principal.
\end{ex}

\begin{rem} \label{rem:Main Theorem} We give several further remarks on this result.
\begin{enumerate}
\item 

For the non-monomial algebra $R=R_2$ in Example \ref{E:Goodleftapproximation} (b), the formula \eqref{E:RingelDualityFormula} from the theorem fails but the following Ringel-duality formula holds
\begin{align} \label{E:AlternativeFormula}
\mathfrak{R}(\End_R(\PI(R))) \cong  \End_{R^{\op}}(\PI(R^{\op}))^{\op}.
\end{align}
For ideally ordered monomial algebras this formula coincides with the formula \eqref{E:RingelDualityFormula} above. Unfortunately, we were not able to find a more general setup where the formula \eqref{E:AlternativeFormula} works.

\item 
Kn{\"o}rrer invariant algebras \cite[Section 6.4.]{KalckKarmazyn}, see Example \ref{ex:ideallyordered} (3) and Section \ref{Section:HillePloog}, and truncated free algebras $R=k\langle x_1, \ldots, x_l \rangle/ (x_1, \ldots, x_l)^m$ satisfy the additional condition imposed in (c).

\item
The statement that $D^b(\sub(R)) \cong D^b(\fac(R^{\dagger}))$ is related to Ringel's \cite[Remark before Corollary 2.2]{RingelRepDim}. It would be interesting to see in what generality this equivalence holds.

We observe that it holds for $k\langle x, y, z \rangle/(p, zx, xy, zy, yz, z^2)$ where $p$ runs over all paths of length $3$, which is not ideally ordered but in which every principal left ideal is isomorphic to a monomial ideal. Indeed, in this case the equivalence is given by a tilting module which is obtained by mutating the characteristic tilting module (for the quasi-hereditary algebra structure defined by the ideal layer function) once.

\item 
Consider $R=k\langle x, y \rangle/ (x^3, y^3, y^2 x, y x^2, xy)$, which is an ideally ordered finite dimensional local monomial algebra. Then there is a surjection $Rx \to Ry$ but $Ry$ does not include into $Rx$. One can check that the order
$R < Ry < Rx < Rx^2$ on indecomposable submodules of $R$ defines a (left but not right strongly) quasi-hereditary structure on $E_R:=\End_R(\SUB(R))$. In particular, in this case the ideal order is not the unique quasi-hereditary order.  

\item 

Part (c) can fail if $R$ is not local (even if all the other conditions are satisfied). Indeed consider for example the algebra $R=kQ/J^2$ where 
\[
\begin{tikzpicture} []
\node (c) at (-1.3,0) {$Q:=$};
\node (C1) at (0,0)  {$1$};
\node (C2) at (1.5,0)  {$2$};
\draw [->] (C1) to node[below]  {\scriptsize{$$}} (C2);
\draw [->, looseness=18, in=142, out=218,loop] (C1) to node[right] {$$} (C1);
\end{tikzpicture}
\]
and $J$ is the ideal generated by all arrows. Then $R$ is ideally ordered and for every surjection between principal left ideals $\Gamma \to \Lambda$ there is an inclusion $\Lambda \to \Gamma$. The order $P_2 < P_1 < S_1$ defines a quasi-hereditary structure on $E_R^{\PI}=\End_R(\PI(R))$ which is not left strongly quasi-hereditary. Hence, it differs from the quasi-hereditary structure defined by the ideal layer function (where $P_2=P_1<S_1$), and there is no unique quasi-hereditary structure in this case.

\item
It is true that $R$ is ideally ordered iff $R^{\op}$ is ideally ordered, and  using this fact one can also prove the theorem without relying on Ringel's result \cite[Theorem 1.1]{RingelRepDim}.
\end{enumerate}
\end{rem}

\section{Applications and examples} 
We discuss some relationships between Theorem \ref{T:Main} and several classes of algebras that have been studied in separate work.

\subsection{Hille \& Ploog's algebras} \label{Section:HillePloog}
The results of this paper were originally motivated by an investigation in \cite{KalckKarmazyn} of a class of geometrically inspired quasi-hereditary algebras introduced by Hille and Ploog \cite{HillePloog} for which the Ringel duality formula has a geometric interpretation, and we briefly recall this geometric set up and these algebras below.

As the geometric background, consider a type $A_n$ configuration of intersecting rational curves $C_1, \dots, C_n$ in a smooth, rational, projective surface $X$ with negative self intersection numbers $C_i \cdot C_i =:-\alpha_i\le -2$. Starting with this data, Hille and Ploog consider the full triangulated subcategory
\[
\langle \mathcal{O}_X(-C_1-\dots -C_n), \mathcal{O}_{X}(-C_2- \dots, -C_{n}), \dots, \mathcal{O}_X(-C_n), \mathcal{O}_X \rangle \subset D^b(\Coh \, X),
\] where we recall that $\mathcal{O}_X(-D)$ denotes the line bundle occurring as the ideal sheaf of an effective divisor $D \subset X$.  Hille \& Ploog show that this subcategory carries an (exact) tilting object $\Lambda$. To do this they make use of universal (co)extensions, see \cite{DlabRingel92TsukubaConferenceProc} and also \cite{HillePerling} for the special case of vector bundles on a rational surface. We briefly recall the definition in this setting. 

\begin{defn}
Consider an ordered pair of vector bundles $\mathcal{E}_1, \mathcal{E}_2$ on a smooth projective rational surface $X$. Their \emph{universal (co)extension} is defined to be the vector bundle occurring in the middle of the short exact sequence
\begin{align*}
\mathcal{E}_2 \otimes \Ext^1_X(\mathcal{E}_1,\mathcal{E}_2)^\dagger \rightarrow \mathcal{F} \rightarrow  \mathcal{E}_1 \tag{extension}\\
\mathcal{E}_2 \rightarrow \mathcal{F} \rightarrow  \mathcal{E}_1  \otimes \Ext^1_X(\mathcal{E}_1,\mathcal{E}_2) \tag{coextension}
\end{align*}
where both sequences are determined by the identity element in $\End(\Ext^1(\mathcal{E}_1,\mathcal{E}_2)) \cong \Ext_X(\mathcal{E}_1,\mathcal{E}_2) \otimes \Ext_X(\mathcal{E}_1,\mathcal{E}_2)^\dagger$.
\end{defn}

Hille and Ploog show that
\[
\mathbb{E}:=\left(\mathcal{O}_X(-C_1-\dots -C_n), \mathcal{O}_{X}(-C_2- \dots -C_{n}), \dots, \mathcal{O}_X(-C_n), \mathcal{O}_X \right)
\] 
is an exceptional sequence of line bundles and that iterated universal extension along this sequence produces a tilting bundle $\Lambda$, see \cite[Section 2]{HillePloog}.

This defines a corresponding algebra 
\[
\Lambda_{[\alpha_1, \dots, \alpha_n]}:=\End_X(\Lambda)^{\op},
\]
where we assume that $\Lambda$ is taken to be a basic representative of the tilting object. These algebras are quasi-hereditary by construction.

We note that the algebra depends on the choice of consecutive ordering for the labelling of the curves and that there are two choices, $C_1, \dots, C_n$ or $C_n, \dots, C_1$, for the same geometric set up that produce two different algebras $\Lambda_{[\alpha_1, \dots, \alpha_n]}$ and $\Lambda_{[\alpha_n, \dots, \alpha_1]}$. This phenomenon is explained by the following result.
\begin{prop} \label{Prop:RingelforHillePloog} There is an isomorphism of algebras
\[
\mathfrak{R}(\Lambda_{[\alpha_1, \dots, \alpha_n]}) \cong  \Lambda_{[\alpha_n, \dots, \alpha_1]}^{\op}.
\]
\end{prop}

\begin{rem}\label{R:HPisKnorrer}

The algebra $\Lambda_{[\alpha_1, \dots, \alpha_n]}$ can in fact be realised in the form $E_{R}$ where $R$ is an ideally ordered monomial Kn\"{o}rrer invariant algebra, as we describe below. Then Proposition \ref{Prop:RingelforHillePloog} is an consequence of Theorem \ref{T:Main}. However, the following alternative, short, geometric proof was explained to us by Agnieszka Bodzenta; indeed it was the existence of a Ringel duality formula in this special case that inspired the representation-theoretic generalisation in this paper. Work of Bodzenta and Bondal also realises a Ringel duality associated to birational morphisms of smooth surfaces by gluing t-structures with reversed orderings, see \cite{BodzentaBondal}.
\end{rem}

\begin{proof}
Let $X$ be a smooth, rational, projective surface containing a type $A_n$ configuration of rational curves with self intersection numbers $\bm{\alpha}:=[\alpha_1, \dots \alpha_n]$. Consider the exceptional sequence $\mathbb{E}$ in the Hom-finite abelian category $\Coh(X)$. By definition, $\Lambda_{\bm{\alpha}}:=\End_X(\Lambda)^{\op}$, where $\Lambda \in \Coh(X)$ is obtained from $\mathbb{E}$ by taking iterated universal extensions and by passing to a basic representative, see \cite[Section 2.3]{KalckKarmazyn}. On the other hand, taking iterated universal coextensions of $\mathbb{E}$ yields $T \in \Coh(X)$ (again we replace this by a basic version if necessary) and it follows from \cite[paragraph above Proposition 3.1.]{DlabRingel92TsukubaConferenceProc} that  there is an algebra isomorphism
\begin{align}\label{E:Ringeldual}
\mathfrak{R}(\End_X(\Lambda)) \cong \End_X(T)^{\op}
\end{align}
where $\mathfrak{R}(\End_X(\Lambda))$ denotes the Ringel-dual of $\End_X(\Lambda)$. More precisely, since $\mathbb{E}$ is standardisable, Dlab \& Ringel \cite[Theorem 2]{DlabRingel92TsukubaConferenceProc} show that $\Hom_X(\Lambda, -) \colon {\mathcal Filt}(\mathbb{E}) \to {\mathcal Filt}(\Delta_{\Lambda_{\bm{\alpha}}})$ defines an exact equivalence sending $\mathbb{E}$ to the sequence of standard modules $\Delta_{\Lambda_{\bm{\alpha}}}$. By Ringel \cite[p. 217 and Proposition 2]{Ringel}, the characteristic tilting module $T_{\Lambda_{\bm{\alpha}}} \in \mod-\Lambda_{\bm{\alpha}}$ is obtained from $\Delta_{\Lambda_{\bm{\alpha}}}$ by iterated universal coextensions (and passing to a basic module if necessary). In particular, the exact equivalence $\Hom_X(\Lambda, -)$ sends $T$ to $T_{\Lambda_{\bm{\alpha}}}$. Combining this with definition of the Ringel dual we see
\[
\mathfrak{R}(\Lambda_{\bm{\alpha}}):=\End_{\Lambda_{\bm{\alpha}}}(T_{\Lambda_{\bm{\alpha}}})^{\op} \cong \End_X(T)^{\op}
\]
as claimed.

Now consider the duality
\begin{align*}
\ddagger: & D(\QCoh(X))  \rightarrow D(\QCoh(X)) \\
& \mathcal{E}^{\ddagger}  :=  R{\mathcal Hom}_X(\mathcal{E} \otimes_X \mathcal{O}(-C_1 - C_2 - \cdots -C_n), \mathcal{O})
\end{align*}
Then $\mathbb{E}^{\ddagger}=(\mathcal{O}(-C_1 - C_2 - \cdots -C_n), \mathcal{O}(-C_1 - C_2 - \cdots -C_{n-1}), \ldots,  \mathcal{O}(-C_1), \mathcal{O})$ and $T^{\ddagger}$ is obtained from this sequence by iterated universal extensions. By definition, $\Lambda_{[\alpha_n, \dots, \alpha_1]} \cong \End_X(T^{\ddagger})^{\op}$. Since $\ddagger$ is a duality, $\Lambda_{[\alpha_n, \dots, \alpha_1]}  \cong \End_X(T^{\ddagger})^{\op} \cong \End_X(T)$. In combination with \eqref{E:Ringeldual} this completes the proof.
\end{proof}

\begin{rem} \label{R:QH opposite algebras}
We note that there is a change in conventions for compositions of morphisms between this paper and \cite{KalckKarmazyn}. This corresponds to exchanging algebras with their opposite algebras, or left modules with right modules. The effect this has on the quasi-hereditary structure and Ringel duality is as follows: if $A$ is a quasi-hereditary algebra with defining layer function $L$ and characteristic tilting module $T$, then $T^{\dagger}$ is the characteristic tilting module for $A^\op$ where $\dagger:A-\mod \rightarrow A^{\op}-\mod$ denotes the standard $k$-duality and the layer function on $A^{\op}$ on is $L^{\dagger}$ defined by $L^{\dagger}(S^{\dagger}):=L(S)$. In particular, $\mathfrak{R}(A^{\op}) \cong \mathfrak{R}(A)^{\op}$.
\end{rem}

We briefly recap how the algebras $\Lambda$ defined by Hille and Ploog fit into the general setup of Theorem \ref{T:Main}. To do so we recall the definition of the Hirzebruch-Jung continued fraction expansion, the Kn\"{o}rrer invariant algebras $K_{r,a}$, and a description of the form $\Lambda_{\bm{\alpha}} \cong E_{K_{r,a}}$.
\begin{defn}
For coprime integers $0<a<r$ the Hirzebruch-Jung continued fraction $[\alpha_1, \dots, \alpha_n]$ is the collection of integers $\alpha_i \ge 2$ defined by 
\[
\frac{r}{a} = \alpha_1 - \cfrac{1}{\alpha_{2}
          -\cfrac{1}{\dots - \cfrac{1}{\alpha_n} } }.
\]
\end{defn}

\begin{defn}[{\cite[Definitions 4.6 \& 6.20 and Corollary 6.27]{KalckKarmazyn}}]
For coprime integers $0<a<r$ the \emph{Kn\"{o}rrer invariant algebra} $K_{r,a}$ is defined to be
\[
K_{r,a}:= 
\frac{\mathbb{C} \langle z_1 \dots, z_l \rangle}
{\left\langle 
{\begin{array}{c}
z_i \left(  z_i^{\beta_i-2} \right) \left(  z_{i+1}^{\beta_{i+1}-2} \right)  \dots {\left(  z_{j-1}^{\beta_{j-1}-2} \right)} \left(  z_j^{\beta_j-2} \right) z_j=0 \text{ for $j \le i$} \\
z_i z_j=0 \text{ if $i<j$}
\end{array}}
\right\rangle}
\]
where the parameters $l \ge 1$ and the $\beta_i \ge 2$ are  defined by the Hirzebruch-Jung continued fraction expansion $[\beta_1, \dots, \beta_l]$ for the fraction $r/(r-a)$.
\end{defn}

We recall that the results of \cite[Section 6.4]{KalckKarmazyn} describe the monomial ideal structure on $K_{r,a}$, and in particular combining \cite[Theorem 6.26]{KalckKarmazyn} and \cite[Propositions 6.22 and 6.24]{KalckKarmazyn} yields the following result.

\begin{prop} \label{P:HPisKnorrer}
The Kn\"{o}rrer invariant algebra $K_{r,a}$ is an ideally order monomial algebra and there is an isomorphism of quasi-hereditary algebras
\[
\Lambda_{[\alpha_1, \dots, \alpha_n]} \cong E_{K_{r,a}}
\]
where $[\alpha_1, \dots, \alpha_n]$ is defined by the Hirzebruch-Jung continued fraction expansion of $r/a$.
\end{prop}
Suppose that $aa' \equiv 1$ mod $r$. If $r/a=[\alpha_1, \dots, \alpha_n]$, then $r/a'=[\alpha_1, \dots, \alpha_n]$. Similarly, if $r/(r-a)=[\beta_1, \dots, \beta_l]$, then $r/(r-a') = [\beta_l, \dots, \beta_1]$. Using this result it can be seen from the explicit definition of $K_{r,a}$ that $K_{r,a'} \cong K_{r,a}^\op$. As a result $\Lambda_{[\alpha_n, \dots, \alpha_1]} \cong E_{K_{r,a}^\op}$ by Proposition \ref{P:HPisKnorrer}, and hence Theorem \ref{T:Main} is a generalisation of Proposition \ref{Prop:RingelforHillePloog}.

\subsection{Example of an application of the Ringel duality formula.} \label{Section:HPExample}
In this section we consider as an example the pair of algebras $\Lambda_{[3,2]}$ and $\Lambda_{[2,3]}$. After giving explicit presentations, we discuss their relationship via Ringel duality, their construction from related Kn\"{o}rrer invariant algebras, and explicitly list the distinguished modules in their quasi-hereditary structures in order to verify the Ringel duality formula.

Firstly,  by \cite[Prop 6.18]{KalckKarmazyn} the algebras $\Lambda_{[3,2]}$ and $\Lambda_{[2,3]}$ can respectively be presented as the path algebra of the following quivers with relations:
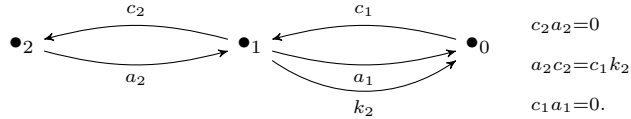
\begin{align*}
\begin{aligned}
\begin{tikzpicture}
\node (C3) at (6,0) {$\bullet_{0}$};
\node (C2) at (3,0)  {$\bullet_{1}$};
\node (C1) at (0,0)  {$\bullet_{2}$};
\draw [->,bend right=15] (C2) to node[above] {$\scriptstyle{c_{2}}$} (C1);
\draw [->,bend right=15] (C1) to node[below] {$\scriptstyle{a_2}$} (C2);
\draw [->,bend right=15] (C3) to node[above] {$\scriptstyle{c_{1}}$} (C2);
\draw [->,bend right=15] (C2) to node[below] {$\scriptstyle{a_1}$} (C3);
\draw [->,bend right=35] (C2) to node[below] {$\scriptstyle{k_2}$} (C3);
\end{tikzpicture}
\end{aligned}
\quad
\begin{aligned}
& \scriptstyle{ c_2a_2=0}\\
& \scriptstyle{ a_2c_2=c_1k_2} \\
&\scriptstyle{ c_1a_1=0.}
\end{aligned}
\end{align*}
and 
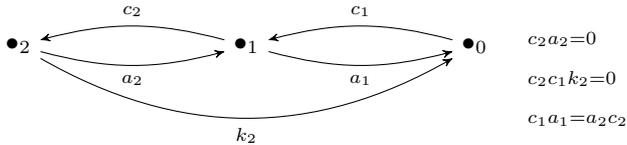
\begin{align*}
\begin{aligned}
\begin{tikzpicture}
\node (C3) at (6,0) {$\bullet_{0}$};
\node (C2) at (3,0)  {$\bullet_{1}$};
\node (C1) at (0,0)  {$\bullet_{2}$};
\draw [->,bend right=15] (C2) to node[above] {$\scriptstyle{c_{2}}$} (C1);
\draw [->,bend right=15] (C1) to node[below] {$\scriptstyle{a_2}$} (C2);
\draw [->,bend right=15] (C3) to node[above] {$\scriptstyle{c_{1}}$} (C2);
\draw [->,bend right=15] (C2) to node[below] {$\scriptstyle{a_1}$} (C3);
\draw [->,bend right=30] (C1) to node[below] {$\scriptstyle{k_2}$} (C3);
\end{tikzpicture}
\end{aligned}
\quad
\begin{aligned}
& \scriptstyle{ c_2a_2=0}\\
& \scriptstyle{ c_2c_1k_2=0} \\
&\scriptstyle{ c_1a_1=a_2c_2.}
\end{aligned}
\end{align*}

Secondly, the Ringel duality formula of Proposition \ref{Prop:RingelforHillePloog} states that
\[
\mathfrak{R}(\Lambda_{[3,2]}) \cong (\Lambda_{[2,3]})^{\op}.
\]

Thirdly, by Proposition \ref{P:HPisKnorrer} the corresponding Kn\"{o}rrer invariant algebras are
\[
K_{[3,2]}:=K_{5,2} \cong \frac{\mathbb{C}\langle z_1,z_2 \rangle }{(z_1^2, z_2^3, z_1z_2,z_2^2z_1) } \quad \text{ and } \quad K_{[2,3]}:=K_{5,3} \cong \frac{\mathbb{C}\langle z_1,z_2 \rangle }{(z_1^3, z_2^2, z_1z_2,z_2z_1^2) }
\]
and these can be presented via the following monomial diagrams
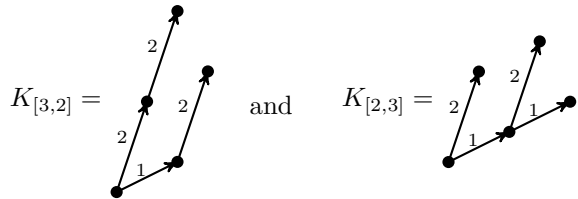
\begin{center}
\begin{align*}
\begin{aligned}
\begin{tikzpicture}[scale = 0.8]
\node (A) at (-1,1.5) {$K_{[3,2]}=$};
\fill[color=black] (0,0) circle (0.1cm);
\fill[color=black] (0.5,1.5) circle (0.1cm);
\fill[color=black] (1,0.5) circle (0.1cm);
\fill[color=black] (1.5,2) circle (0.1cm);
\fill[color=black] (1,3) circle (0.1cm);
\draw[thick,->] (0,0) -- (1,0.5)  node[midway, above left=-5pt]{$\scriptstyle{1}$};
\draw[thick,->] (0,0) -- (0.5,1.5)  node[midway, above left=-3pt]{$\scriptstyle{2}$};
\draw[thick,->] (0.5,1.5) -- (1,3)  node[midway, above left=-3pt]{$\scriptstyle{2}$};
\draw[thick,->] (1,0.5) -- (1.5,2)  node[midway, above left=-3pt]{$\scriptstyle{2}$};
\end{tikzpicture}
\end{aligned}
\quad \text{ and } \quad 
\begin{aligned}
 \begin{tikzpicture}[scale = 0.8]
 \node (A) at (-1,1) {$K_{[2,3]}=$};
\fill[color=black] (0,0) circle (0.1cm);
\fill[color=black] (0.5,1.5) circle (0.1cm);
\fill[color=black] (1,0.5) circle (0.1cm);
\fill[color=black] (1.5,2) circle (0.1cm);
\fill[color=black] (2,1) circle (0.1cm);
\draw[thick,->] (0,0) -- (1,0.5)  node[midway, above left=-5pt]{$\scriptstyle{1}$};
\draw[thick,->] (1,0.5) -- (2,1)  node[midway, above left=-5pt]{$\scriptstyle{1}$};
\draw[thick,->] (0,0) -- (0.5,1.5)  node[midway, above left=-3pt]{$\scriptstyle{2}$};
\draw[thick,->] (1,0.5) -- (1.5,2)  node[midway, above left=-3pt]{$\scriptstyle{2}$};
\end{tikzpicture}
\end{aligned}
\end{align*}
\end{center}
where the nodes represent the monomial basis of $K_{r,a}$ with the root of the tree representing 1 and the arrows labelled $i$ representing left multiplication by $z_i$ of the node at the source equalling the node at the target. Using these monomial diagrams one can show that $K_{[3,2]} \cong K_{[2,3]}^{\op}$ and to calculate all the left monomial ideals. The left monomial ideals for $K_{[3,2]}$ are $M_0 \cong (1)$, \, $M_1 \cong (z_1)$ and $M_2 \cong (z_2 z_1)$ and the left monomial ideals for $K_{[2,3]}$ are $N_0 \cong (1)$, $N_1  \cong (z_1)$ and $N_2 \cong (z_1^2)$.  These can represented pictorially as subsets of the monomial diagrams by
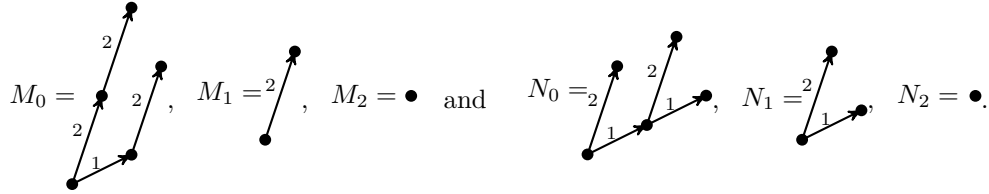
\begin{center}
\begin{align*}
\begin{aligned}
\begin{tikzpicture}[scale = 0.78]
\node (A) at (-0.5,1.5) {$M_0=$};
\fill[color=black] (0,0) circle (0.1cm);
\fill[color=black] (0.5,1.5) circle (0.1cm);
\fill[color=black] (1,0.5) circle (0.1cm);
\fill[color=black] (1.5,2) circle (0.1cm);
\fill[color=black] (1,3) circle (0.1cm);
\draw[thick,->] (0,0) -- (1,0.5)  node[midway, above left=-5pt]{$\scriptstyle{1}$};
\draw[thick,->] (0,0) -- (0.5,1.5)  node[midway, above left=-3pt]{$\scriptstyle{2}$};
\draw[thick,->] (0.5,1.5) -- (1,3)  node[midway, above left=-3pt]{$\scriptstyle{2}$};
\draw[thick,->] (1,0.5) -- (1.5,2)  node[midway, above left=-3pt]{$\scriptstyle{2}$};
\end{tikzpicture}
\end{aligned}
, \,
\begin{aligned}
\begin{tikzpicture}[scale = 0.78]
\node (A) at (1.9,0.8) {$M_1=$};
\fill[color=black] (2.5,0) circle (0.1cm);
\fill[color=black] (3,1.5) circle (0.1cm);
\draw[thick,->] (2.5,0) -- (3,1.5)  node[midway, above left=-3pt]{$\scriptstyle{2}$};
\end{tikzpicture}
\end{aligned}
, \, 
\begin{aligned}
\begin{tikzpicture}[scale = 0.78]
\node (A) at (3.7,0.1) {$M_2=$};
\fill[color=black] (4.5,0.1) circle (0.1cm);
\end{tikzpicture}
\end{aligned}
\quad \text{and} \quad
\begin{aligned}
\begin{tikzpicture}[scale = 0.78]
\node (A) at (-0.5,1.1) {$N_0=$};
\fill[color=black] (0,0) circle (0.1cm);
\fill[color=black] (0.5,1.5) circle (0.1cm);
\fill[color=black] (1,0.5) circle (0.1cm);
\fill[color=black] (1.5,2) circle (0.1cm);
\fill[color=black] (2,1) circle (0.1cm);
\draw[thick,->] (0,0) -- (1,0.5)  node[midway, above left=-5pt]{$\scriptstyle{1}$};
\draw[thick,->] (1,0.5) -- (2,1)  node[midway, above left=-5pt]{$\scriptstyle{1}$};
\draw[thick,->] (0,0) -- (0.5,1.5)  node[midway, above left=-3pt]{$\scriptstyle{2}$};
\draw[thick,->] (1,0.5) -- (1.5,2)  node[midway, above left=-3pt]{$\scriptstyle{2}$};
\end{tikzpicture}
\end{aligned}
, \,
\begin{aligned}
\begin{tikzpicture}[scale = 0.78]
\node (A) at (2,0.7) {$N_1=$};
\fill[color=black] (2.5,0) circle (0.1cm);
\fill[color=black] (3,1.5) circle (0.1cm);
\fill[color=black] (3.5,0.5) circle (0.1cm);
\draw[thick,->] (2.5,0) -- (3.5,0.5)  node[midway, above left=-5pt]{$\scriptstyle{1}$};
\draw[thick,->] (2.5,0) -- (3,1.5)  node[midway, above left=-3pt]{$\scriptstyle{2}$};
\end{tikzpicture}
\end{aligned}
, \,
\begin{aligned}
\begin{tikzpicture}[scale = 0.78]
\node (A) at (3.7,0.1) {$N_2=$};
\fill[color=black] (4.5,0.1) circle (0.1cm);
\end{tikzpicture}
\end{aligned}.
\end{align*}
\end{center}
It is explicit that $\Lambda_{[3,2]} \cong E_{K_{[3,2]}} :=\End_{K_{[3,2]}}(\oplus M_i)$ and $ \Lambda_{[2,3]}  \cong E_{K_{[2,3]}}:=\End_{K_{[2,3]}}(\oplus N_i)$.

In order to explicitly verify the Ringel duality formula in this case we first describe the quasi-hereditary structure by calculating the projective $P_i$,  injective $I_i$, standard $\Delta_i$, costandard $\nabla_i$, and characteristic tilting $T_i$ objects for each algebra.  We list these modules in the table below in terms of the simples, $S_i$ notated by $i$, occurring in their composition series with the heads written at the top.

\[
\begin{array}{c|ccccc}
\Lambda_{[3,2]}& P_i & I_i  & \Delta_i  & \nabla_i & T_i \\
\hline
i=0
&
\begin{tikzpicture}[baseline=3ex]
\node (C6) at (-0.2,-1.6) {$0$};
\node (C5) at (0.2,-1.6) {$0$};
\node (C4) at (0,-1) {$1$};
\node (C3) at (-0.6,-0.4) {$0$};
\node (C2) at (-0.2,-0.4) {$2$};
\node (C1) at (0.2,-0.4) {$0$};
\node (C0) at (0,0.4)  {$1$};
\node (C00) at (0,1)  {$0$};
\draw [-] (C1) to node[right] {$\scriptstyle{k_2}$} (C0);
\draw [-] (C2) to node[gap] {$\scriptstyle{c_2}$} (C0);
\draw [-] (C3) to node[left] {$\scriptstyle{a_1}$} (C0);
\draw [-] (C4) to node[left] {$\scriptstyle{a_2}$} (C2);
\draw [-] (C4) to node[right] {$\scriptstyle{c_1}$} (C1);
\draw [-] (C6) to node[left] {$\scriptstyle{a_1}$} (C4);
\draw [-] (C5) to node[right] {$\scriptstyle{k_2}$} (C4);
\draw [-] (C0) to node[left] {$\scriptstyle{c_1}$} (C00);
\end{tikzpicture}
 &
\begin{tikzpicture}[baseline=-1ex]
\node (C5) at (0.6,-2.1) {$0$};
\node (C4) at (0,-1.5) {$1$};
\node (C3) at (0.2,-0.9) {$0$};
\node (C2) at (-0.2,-0.9) {$2$};
\node (C1) at (0,-0.3) {$1$};
\node (C0) at (0,0.3)  {$0$};
\draw [-] (C4) to node[left] {$\scriptstyle{a_2}$} (C2);
\draw [-] (C4) to node[right] {$\scriptstyle{c_1}$} (C3);
\draw [-] (C2) to node[left] {$\scriptstyle{c_2}$} (C1);
\draw [-] (C3) to node[right] {$\scriptstyle{k_2}$} (C1);
\draw [-] (C1) to node[left] {$\scriptstyle{c_1}$} (C0);
\node (C4') at (1.2,-1.5) {$1$};
\node (C3') at (1.4,-0.9) {$0$};
\node (C2') at (1.0,-0.9) {$2$};
\node (C1') at (1.2,-0.3) {$1$};
\node (C0') at (1.2,0.3)  {$0$};
\draw [-] (C4') to node[left] {$\scriptstyle{a_2}$} (C2');
\draw [-] (C4') to node[right] {$\scriptstyle{c_1}$} (C3');
\draw [-] (C2') to node[left] {$\scriptstyle{c_2}$} (C1');
\draw [-] (C3') to node[right] {$\scriptstyle{k_2}$} (C1');
\draw [-] (C1') to node[left] {$\scriptstyle{c_1}$} (C0');
\draw [-] (C5) to node[below left] {$\scriptstyle{k_2}$} (C4);
\draw [-] (C5) to node[below right] {$\scriptstyle{c_2}$} (C4');
\end{tikzpicture}
&
0
&
0
& 
0
\\
\\
i=1 
&
\begin{tikzpicture}[baseline=2ex]
\node (C6) at (-0.2,-1.6) {$0$};
\node (C5) at (0.2,-1.6) {$0$};
\node (C4) at (0,-1) {$1$};
\node (C3) at (-0.6,-0.4) {$0$};
\node (C2) at (-0.2,-0.4) {$2$};
\node (C1) at (0.2,-0.4) {$0$};
\node (C0) at (0,0.4)  {$1$};
\draw [-] (C1) to node[right] {$\scriptstyle{k_2}$} (C0);
\draw [-] (C2) to node[gap] {$\scriptstyle{c_2}$} (C0);
\draw [-] (C3) to node[left] {$\scriptstyle{a_1}$} (C0);
\draw [-] (C4) to node[left] {$\scriptstyle{a_2}$} (C2);
\draw [-] (C4) to node[right] {$\scriptstyle{c_1}$} (C1);
\draw [-] (C6) to node[left] {$\scriptstyle{a_1}$} (C4);
\draw [-] (C5) to node[right] {$\scriptstyle{k_2}$} (C4);
\end{tikzpicture}
& 
\begin{tikzpicture}[baseline=2ex]
\node (C4) at (0,-1.5) {$1$};
\node (C3) at (0.2,-0.9) {$0$};
\node (C2) at (-0.2,-0.9) {$2$};
\node (C1) at (0,-0.3) {$1$};
\node (C0) at (0,0.3)  {$0$};
\draw [-] (C4) to node[left] {$\scriptstyle{a_2}$} (C2);
\draw [-] (C4) to node[right] {$\scriptstyle{c_1}$} (C3);
\draw [-] (C2) to node[left] {$\scriptstyle{c_2}$} (C1);
\draw [-] (C3) to node[right] {$\scriptstyle{k_2}$} (C1);
\draw [-] (C1) to node[left] {$\scriptstyle{c_1}$} (C0);
\end{tikzpicture}
& 
\begin{tikzpicture}[baseline=-2ex]
\node (C3) at (0.2,-0.9) {$0$};
\node (C2) at (-0.2,-0.9) {$0$};
\node (C1) at (0,-0.3) {$1$};
\draw [-] (C2) to node[left] {$\scriptstyle{a_1}$} (C1);
\draw [-] (C3) to node[right] {$\scriptstyle{k_2}$} (C1);
\end{tikzpicture}
&
\begin{tikzpicture}[baseline=2ex]
\node (C1) at (0,-0.3) {$1$};
\node (C0) at (0,0.3)  {$0$};
\draw [-] (C1) to node[left] {$\scriptstyle{c_1}$} (C0);
\end{tikzpicture} 
&
\begin{tikzpicture}[baseline=-3ex]
\node (C6) at (-0.2,-1.6) {$0$};
\node (C5) at (0.2,-1.6) {$0$};
\node (C4) at (0,-1) {$1$};
\node (C1) at (0,-0.4) {$0$};
\draw [-] (C4) to node[right] {$\scriptstyle{c_1}$} (C1);
\draw [-] (C6) to node[left] {$\scriptstyle{a_1}$} (C4);
\draw [-] (C5) to node[right] {$\scriptstyle{k_2}$} (C4);
\end{tikzpicture}
\\
\\
i=2 
&
\begin{tikzpicture}[baseline=1.5ex]
\node (C3) at (0.2,-0.9) {$0$};
\node (C2) at (-0.2,-0.9) {$0$};
\node (C1) at (0,-0.3) {$1$};
\node (C0) at (0,0.3)  {$2$};
\draw [-] (C2) to node[left] {$\scriptstyle{a_1}$} (C1);
\draw [-] (C3) to node[right] {$\scriptstyle{k_2}$} (C1);
\draw [-] (C1) to node[left] {$\scriptstyle{a_2}$} (C0);
\end{tikzpicture}
&
\begin{tikzpicture}[baseline=1.5ex]
\node (C2) at (0,-0.9) {$2$};
\node (C1) at (0,-0.3) {$1$};
\node (C0) at (0,0.3)  {$0$};
\draw [-] (C2) to node[left] {$\scriptstyle{c_2}$} (C1);
\draw [-] (C1) to node[left] {$\scriptstyle{c_1}$} (C0);
\end{tikzpicture} 
&
\begin{tikzpicture}[baseline=1.5ex]
\node (C3) at (0.2,-0.9) {$0$};
\node (C2) at (-0.2,-0.9) {$0$};
\node (C1) at (0,-0.3) {$1$};
\node (C0) at (0,0.3)  {$2$};
\draw [-] (C2) to node[left] {$\scriptstyle{a_1}$} (C1);
\draw [-] (C3) to node[right] {$\scriptstyle{k_2}$} (C1);
\draw [-] (C1) to node[left] {$\scriptstyle{a_2}$} (C0);
\end{tikzpicture}
&
\begin{tikzpicture}[baseline=1.5ex]
\node (C2) at (0,-0.9) {$2$};
\node (C1) at (0,-0.3) {$1$};
\node (C0) at (0,0.3)  {$0$};
\draw [-] (C2) to node[left] {$\scriptstyle{c_2}$} (C1);
\draw [-] (C1) to node[left] {$\scriptstyle{c_1}$} (C0);
\end{tikzpicture}
&
\begin{tikzpicture}[baseline=3ex]
\node (C6) at (-0.2,-1.6) {$0$};
\node (C5) at (0.2,-1.6) {$0$};
\node (C4) at (0,-1) {$1$};
\node (C3) at (-0.6,-0.4) {$0$};
\node (C2) at (-0.2,-0.4) {$2$};
\node (C1) at (0.2,-0.4) {$0$};
\node (C0) at (0,0.4)  {$1$};
\node (C00) at (0,1)  {$0$};
\draw [-] (C1) to node[right] {$\scriptstyle{k_2}$} (C0);
\draw [-] (C2) to node[gap] {$\scriptstyle{c_2}$} (C0);
\draw [-] (C3) to node[left] {$\scriptstyle{a_1}$} (C0);
\draw [-] (C4) to node[left] {$\scriptstyle{a_2}$} (C2);
\draw [-] (C4) to node[right] {$\scriptstyle{c_1}$} (C1);
\draw [-] (C6) to node[left] {$\scriptstyle{a_1}$} (C4);
\draw [-] (C5) to node[right] {$\scriptstyle{k_2}$} (C4);
\draw [-] (C0) to node[left] {$\scriptstyle{c_1}$} (C00);
\end{tikzpicture}
\\
\end{array}
\]
\[
\begin{array}{c|ccccc}
\Lambda_{[2,3]} &  P_i & I_i & \Delta_i &  \nabla_i & T_i \\
\hline
i=0 
&
\begin{tikzpicture}[baseline=20ex]
\node (C0) at (0,3.6) {$0$};
\node (C1) at (0,3) {$1$};
\node (C2) at (-0.2,2.4) {$0$};
\node (C3) at (0.2,2.4) {$2$};
\node (C4) at (0,1.6) {$1$};
\node (C5) at (0.4,1.6) {$0$};
\node (C6) at (0,1) {$0$};
\node (C7) at (0.4,1) {$1$};
\node (C8) at (0.4,0.4) {$0$};
\draw [-] (C0) to node[left] {$\scriptstyle{c_1}$} (C1);
\draw [-] (C1) to node[left] {$\scriptstyle{a_1}$} (C2);
\draw [-] (C1) to node[right] {$\scriptstyle{c_2}$} (C3);
\draw [-] (C2) to node[left] {$\scriptstyle{c_1}$} (C4);
\draw [-] (C3) to node[gap] {$\scriptstyle{a_2}$} (C4);
\draw [-] (C3) to node[right] {$\scriptstyle{k_2}$} (C5);
\draw [-] (C4) to node[left] {$\scriptstyle{a_1}$} (C6);
\draw [-] (C5) to node[right] {$\scriptstyle{c_1}$} (C7);
\draw [-] (C7) to node[right] {$\scriptstyle{a_1}$} (C8);
\end{tikzpicture}
&
\begin{tikzpicture}[baseline=16ex]
\node (C0) at (0,0) {$0$};
\node (C1) at (-0.2,0.6) {$1$};
\node (C2) at (0.2,0.6) {$2$};
\node (C3) at (-0.4,1.4) {$0$};
\node (C4) at (0,1.4) {$2$};
\node (C5) at (0.4,1.4) {$1$};
\node (C6) at (-0.6,2.2) {$2$};
\node (C7) at (-0.2,2.2) {$1$};
\node (C8) at (0.6,2.2) {$0$};
\node (C9) at (-0.8,2.8) {$1$};
\node (C10) at (-0.2,2.8) {$0$};
\node (C11) at (-0.8,3.4) {$0$};
\draw [-] (C0) to node[left] {$\scriptstyle{a_1}$} (C1);
\draw [-] (C0) to node[right] {$\scriptstyle{k_2}$} (C2);
\draw [-] (C1) to node[left] {$\scriptstyle{c_1}$} (C3);
\draw [-] (C1) to node[gap] {$\scriptstyle{a_2}$} (C4);
\draw [-] (C2) to node[right] {$\scriptstyle{c_2}$} (C5);
\draw [-] (C3) to node[left] {$\scriptstyle{k_2}$} (C6);
\draw [-] (C3) to node[gap] {$\scriptstyle{a_1}$} (C7);
\draw [-] (C4) to node[right] {$\scriptstyle{c_2}$} (C7);
\draw [-] (C5) to node[right] {$\scriptstyle{c_1}$} (C8);
\draw [-] (C6) to node[left] {$\scriptstyle{c_2}$} (C9);
\draw [-] (C7) to node[right] {$\scriptstyle{c_1}$} (C10);
\draw [-] (C9) to node[left] {$\scriptstyle{c_1}$} (C11);
\end{tikzpicture}
&
0
&
0
& 
0
\\
\\
i=1 
& 
\begin{tikzpicture}[baseline=15ex]
\node (C1) at (0,3) {$1$};
\node (C2) at (-0.2,2.4) {$0$};
\node (C3) at (0.2,2.4) {$2$};
\node (C4) at (0,1.6) {$1$};
\node (C5) at (0.4,1.6) {$0$};
\node (C6) at (0,1) {$0$};
\node (C7) at (0.4,1) {$1$};
\node (C8) at (0.4,0.4) {$0$};
\draw [-] (C1) to node[left] {$\scriptstyle{a_1}$} (C2);
\draw [-] (C1) to node[right] {$\scriptstyle{c_2}$} (C3);
\draw [-] (C2) to node[left] {$\scriptstyle{c_1}$} (C4);
\draw [-] (C3) to node[gap] {$\scriptstyle{a_2}$} (C4);
\draw [-] (C3) to node[right] {$\scriptstyle{k_2}$} (C5);
\draw [-] (C4) to node[left] {$\scriptstyle{a_1}$} (C6);
\draw [-] (C5) to node[right] {$\scriptstyle{c_1}$} (C7);
\draw [-] (C7) to node[right] {$\scriptstyle{a_1}$} (C8);
\end{tikzpicture}
& 
\begin{tikzpicture}[baseline=16ex]
\node (C1) at (-0.2,0.6) {$1$};
\node (C3) at (-0.4,1.4) {$0$};
\node (C4) at (0,1.4) {$2$};
\node (C6) at (-0.6,2.2) {$2$};
\node (C7) at (-0.2,2.2) {$1$};
\node (C9) at (-0.8,2.8) {$1$};
\node (C10) at (-0.2,2.8) {$0$};
\node (C11) at (-0.8,3.4) {$0$};
\draw [-] (C1) to node[left] {$\scriptstyle{c_1}$} (C3);
\draw [-] (C1) to node[gap] {$\scriptstyle{a_2}$} (C4);
\draw [-] (C3) to node[left] {$\scriptstyle{k_2}$} (C6);
\draw [-] (C3) to node[gap] {$\scriptstyle{a_1}$} (C7);
\draw [-] (C4) to node[right] {$\scriptstyle{c_2}$} (C7);
\draw [-] (C6) to node[left] {$\scriptstyle{c_2}$} (C9);
\draw [-] (C7) to node[right] {$\scriptstyle{c_1}$} (C10);
\draw [-] (C9) to node[left] {$\scriptstyle{c_1}$} (C11);
\end{tikzpicture}
&
\begin{tikzpicture}[baseline=5ex]
\node (C0) at (0,0) {$0$};
\node (C1) at (0,0.6) {$1$};
\draw [-] (C0) to node[left] {$\scriptstyle{a_1}$} (C1);
\end{tikzpicture}
&
\begin{tikzpicture}[baseline=24ex]
\node (C0) at (0,3.6) {$0$};
\node (C1) at (0,3) {$1$};
\draw [-] (C0) to node[left] {$\scriptstyle{c_1}$} (C1);
\end{tikzpicture}
&
\begin{tikzpicture}[baseline=10ex]
\node (C5) at (0.4,1.6) {$0$};
\node (C7) at (0.4,1) {$1$};
\node (C8) at (0.4,0.4) {$0$};
\draw [-] (C5) to node[right] {$\scriptstyle{c_1}$} (C7);
\draw [-] (C7) to node[right] {$\scriptstyle{a_1}$} (C8);
\end{tikzpicture}
\\
\\
i=2 &
\begin{tikzpicture}[baseline=10ex]
\node (C3) at (0.2,2.4) {$2$};
\node (C4) at (0,1.6) {$1$};
\node (C5) at (0.4,1.6) {$0$};
\node (C6) at (0,1) {$0$};
\node (C7) at (0.4,1) {$1$};
\node (C8) at (0.4,0.4) {$0$};
\draw [-] (C3) to node[gap] {$\scriptstyle{a_2}$} (C4);
\draw [-] (C3) to node[right] {$\scriptstyle{k_2}$} (C5);
\draw [-] (C4) to node[left] {$\scriptstyle{a_1}$} (C6);
\draw [-] (C5) to node[right] {$\scriptstyle{c_1}$} (C7);
\draw [-] (C7) to node[right] {$\scriptstyle{a_1}$} (C8);
\end{tikzpicture}
&
\begin{tikzpicture}[baseline=20ex]
\node (C6) at (-0.8,2.2) {$2$};
\node (C9) at (-0.8,2.8) {$1$};
\node (C11) at (-0.8,3.4) {$0$};
\draw [-] (C6) to node[left] {$\scriptstyle{c_2}$} (C9);
\draw [-] (C9) to node[left] {$\scriptstyle{c_1}$} (C11);
\end{tikzpicture}
&
\begin{tikzpicture}[baseline=10ex]
\node (C3) at (0.2,2.4) {$2$};
\node (C4) at (0,1.6) {$1$};
\node (C5) at (0.4,1.6) {$0$};
\node (C6) at (0,1) {$0$};
\node (C7) at (0.4,1) {$1$};
\node (C8) at (0.4,0.4) {$0$};
\draw [-] (C3) to node[gap] {$\scriptstyle{a_2}$} (C4);
\draw [-] (C3) to node[right] {$\scriptstyle{k_2}$} (C5);
\draw [-] (C4) to node[left] {$\scriptstyle{a_1}$} (C6);
\draw [-] (C5) to node[right] {$\scriptstyle{c_1}$} (C7);
\draw [-] (C7) to node[right] {$\scriptstyle{a_1}$} (C8);
\end{tikzpicture}
&
\begin{tikzpicture}[baseline=20ex]
\node (C6) at (-0.8,2.2) {$2$};
\node (C9) at (-0.8,2.8) {$1$};
\node (C11) at (-0.8,3.4) {$0$};
\draw [-] (C6) to node[left] {$\scriptstyle{c_2}$} (C9);
\draw [-] (C9) to node[left] {$\scriptstyle{c_1}$} (C11);
\end{tikzpicture}
&
\begin{tikzpicture}[baseline=20ex]
\node (C0) at (0,3.6) {$0$};
\node (C1) at (0,3) {$1$};
\node (C2) at (-0.2,2.4) {$0$};
\node (C3) at (0.2,2.4) {$2$};
\node (C4) at (0,1.6) {$1$};
\node (C5) at (0.4,1.6) {$0$};
\node (C6) at (0,1) {$0$};
\node (C7) at (0.4,1) {$1$};
\node (C8) at (0.4,0.4) {$0$};
\draw [-] (C0) to node[left] {$\scriptstyle{c_1}$} (C1);
\draw [-] (C1) to node[left] {$\scriptstyle{a_1}$} (C2);
\draw [-] (C1) to node[right] {$\scriptstyle{c_2}$} (C3);
\draw [-] (C2) to node[left] {$\scriptstyle{c_1}$} (C4);
\draw [-] (C3) to node[gap] {$\scriptstyle{a_2}$} (C4);
\draw [-] (C3) to node[right] {$\scriptstyle{k_2}$} (C5);
\draw [-] (C4) to node[left] {$\scriptstyle{a_1}$} (C6);
\draw [-] (C5) to node[right] {$\scriptstyle{c_1}$} (C7);
\draw [-] (C7) to node[right] {$\scriptstyle{a_1}$} (C8);
\end{tikzpicture}
\end{array}
\]

Using these descriptions of the characteristic tilting modules, it is a short exercise to verify the Ringel duality formula by direct calculation: 
\[
\begin{tikzpicture}
\node (End) at (-2.8,0) {$\mathfrak{R}(\Lambda_{[3,2]})^\op=\End_{\Lambda_{[3,2]}}(\bigoplus T_i) \cong$};
\node (C3) at (6,0) {$T_{2}$};
\node (C2) at (3,0)  {$T_{1}$};
\node (C1) at (0,0)  {$T_{0}$};
\node (Lam) at (7,0)  {$\cong \Lambda_{[2,3]}$};
\draw [right hook->,bend right=15] (C1) to node {$\scriptstyle{}$} (C2);
\draw [->>,bend right=15] (C2) to node {$\scriptstyle{}$} (C1);
\draw [right hook->,bend right=15] (C2) to node {$\scriptstyle{}$} (C3);
\draw [->>,bend right=15] (C3) to node {$\scriptstyle{}$} (C2);
\draw [right hook->,bend right=25] (C1) to node {$\scriptstyle{}$} (C3);
\end{tikzpicture}
\]
and
\[
\begin{tikzpicture}
\node (End) at (-2.8,0) {$\mathfrak{R}(\Lambda_{[2,3]})^\op=\End_{\Lambda_{[2,3]}}(\bigoplus T_i) \cong$};
\node (C3) at (6,0) {$T_{2}$};
\node (C2) at (3,0)  {$T_{1}$};
\node (C1) at (0,0)  {$T_{0}$};
\node (Lam) at (7,0)  {$\cong \Lambda_{[3,2]}.$};
\draw [right hook->,bend right=15] (C1) to node {$\scriptstyle{}$} (C2);
\draw [->>,bend right=15] (C2) to node {$\scriptstyle{}$} (C1);
\draw [right hook->,bend right=15] (C2) to node {$\scriptstyle{}$} (C3);
\draw [->>,bend right=15] (C3) to node {$\scriptstyle{}$} (C2);
\draw [right hook->,bend right=25] (C2) to node {$\scriptstyle{}$} (C3);
\end{tikzpicture}
\]

\begin{rem}
We observe some further properties of, and relations between, the modules in the tables above. These are all special cases of the general theory developed above.
\begin{enumerate}
\item If $i \leq j$ in the partial order, then there is an inclusion $P_j \subseteq P_i$ (and a projection $I_j \twoheadrightarrow I_i$). This holds for all left (respectively right) strongly quasi-hereditary algebras. In other words, in this situation it is a consequence of Theorem \ref{thm: ideally ordered implies strongly quasi-hereditary}.
\item Every submodule of a standard module $\Delta_i$ or a projective module $P_i$ is filtered by standard modules. This is a consequence of Corollary \ref{C:ClosedSubFac}. Dually, quotients of costandard modules $\nabla_i$ or injective modules $I_i$ are filtered by costandard modules, again by Corollary \ref{C:ClosedSubFac}.
\item For both algebras the only simple costandard module is $\nabla_0$. One can check that the corresponding projective modules $P_0$ are filtered by costandard modules. This illustrates Proposition \ref{P:AefiltCS} in these cases.
\item[$(3^\op)$] For both algebras the only simple standard module is $\Delta_0$. The corresponding injective hulls $I_0$ are not filtered by standard modules. In other words, the algebras $\Lambda_{[3,2]}$ and $\Lambda_{[2,3]}$ are not right ultra strongly quasi-hereditary.
\item The summands $T_i$ of the characteristic tilting module are precisely those indecomposable modules which are both quotients and submodules of the projective module $P_0$, see Theorem \ref{T:Main} (d). In particular, they have head $S_0$ and a socle in $\add(S_0)$.
\end{enumerate}
\end{rem}

\subsection{Auslander-Dlab-Ringel algebras} \label{Section:ADR} Recent results of Conde-Erdmann \cite{CondeErdmann}, and work in Conde's thesis, produce a Ringel duality formula similar to that of Theorem \ref{T:Main} for the class of Auslander-Dlab-Ringel (ADR) algebras.

\begin{defn}
Let $R$ be a finite dimensional algebra of Loewy length $L_R$. Define the additive subcategory
\[
\adr(R):=\add \{ Re/\rad^i \, Re \mid   e \text{ a primitive idempotent and } i=1, \dots, L_R \}
\]
and let $\ADR(R):=\bigoplus_{M \in \text{ind($\adr(R)$)}}M$ be the direct sum of indecomposable elements of the additive category $\adr(R)$ up to isomorphism. Then the associated \emph{ADR algebra} is defined to be
\[
E_R^{\ADR}:=\End_R \left( \ADR(R) \right).
\]
This is the basic algebra Morita equivalent to $\End_R\left( \bigoplus_{i=1}^{L_R} R/ \rad^i \, R \right)$. In particular, the indecompomposable modules in $\adr(R)$ are exactly those of the form $Re/\rad^l Re$ for $e$ a primative idempotent and $1 \le i \le L_{Re}$ where $L_{Re}$ is the Loewy length of $Re$.
\end{defn}

\begin{rem}
We remark that the ADR algebra defined here is the opposite algebra of the ADR algebra defined by Conde and Erdmann in \cite{CondeErdmann}, however the effect on the quasi-hereditary structure is straightforward as is explained in Remark \ref{R:QH opposite algebras}.
\end{rem}

The ADR algebra $E_R^{\ADR}$ is quasi-hereditary for the layer function 
$l(Re_i/\rad^l Re_i):= L_R- l$; this induces the partial ordering
\[
Re_i/\rad^l Re_i < Re_j/\rad^k Re_j \Leftrightarrow l > k
\]
on indecomposable modules in $\adr(R)$. Indeed it is left ultra strongly quasi-hereditary (see \cite[Section 5]{Conde}), and Conde and Erdmann obtain the following Ringel duality formula for ADR algebras satifying a regularity condition; we recall that a module is \emph{rigid} if its radical and socle series coincide.
\begin{thm} \label{T:ADR} Let $R$ be an Artin algebra with Loewy length $L$. If all projective and injective indecomposable $R$-modules are rigid with Loewy length $L$, then
\[
\mathfrak{R}(E_R^{\ADR}) \cong (E_{R^{\op}}^{\ADR})^{\op}.
\]
That is, the Ringel dual of $E_R^{\ADR}$ is isomorphic to the opposite algebra of $E_{R^{\op}}^{\ADR}$.
\end{thm}

This formula looks very similar to the formula in Theorem \ref{T:Main} of this paper. However, in general $E_R \not \cong E_R^{\ADR}$ and there does not appear to be any reason to think the overlap is large.

For example, ADR algebras are not left and right strongly hereditary in general and so not all ADR algebras are in the $E_R$ algebra class. Moreover,  it can be seen that Hille \& Ploog algebras are not always ADR-algebras. Indeed, in the example of Section \ref{Section:HPExample} the modules $R/ \rad^i \, R$ are straightforward to calculate from the monomial diagrams, and the additive category generated by such objects can be seen to coincide with the additive category $\sub(R)\cong \textsf{pi}(R)$ for $R=K_{[3,2]}$ so $E_{R}\cong E_R^{\ADR}$ but not for $R^{\op} \cong K_{[2,3]}$ where $E_R \not \cong E_R^{\ADR}$. 

Indeed, the results of Conde and Erdmann also only describe the Ringel dual of an ADR algebra when the dual is also an ADR algebra. However, as can be seen in the example of Section \ref{Section:HPExample}, there are examples of ADR algebras of the form $E_R$ whose dual is not an ADR algebra but whose Ringel dual can still be described by Theorem \ref{T:Main}:
for $R=K_{[3,2]}$ and $R^{\op} \cong K_{[2,3]}$
\[
\mathfrak{R}(E_{R}^{\ADR}) \cong \mathfrak{R}(E_R) \cong E_{R^{\op}}^\op \not\cong E_{R^{\op}}^{\ADR}.
\]
Indeed, it is also straightforward to calculate the socle and radical filtrations in this example and hence clear to see that $K_{[3,2]}$ is rigid whereas $K_{[2,3]}$ is not.

Whilst these classes of algebras may not be related in general, there are cases which fall into both classes of algebras. Recall the monomial algebras $R:=kQ/J^m$ of Example \ref{ex:ideallyordered} (2) which are ideally ordered and for which $\sub(R)\cong \mathsf{pi}(R) \subset \adr(R)$. In particular, in this case $E_R$ is a corner algebra of $E_R^{\ADR}$: i.e. there is an idempotent $e\in E_R^{\ADR}$ such that $E_R \cong e E_{R}^{\ADR}e$.

\begin{prop} \label{P:ADR=IOM}
Let $Q$ be a finite quiver without sources and $J$ be the two-sided ideal generated by all arrows in $Q$. Then $R:= kQ/J^m$ is an ideally ordered monomial algebra and there is an isomorphism of quasi-hereditary algebras
\[
E_R^{\ADR} \cong E_R.
\]
\end{prop}
\begin{proof}
The algebra $R$ has Loewy length $m$ and, as noted in Example \ref{ex:ideallyordered} (2), any monomial ideal is isomorphic to $Re/\text{rad}^l Re$ for some $l=1, \dots, m$ and some primitive idempotent $e \in R$, hence $R$ is ideally ordered and $\mathsf{pi}(R) \subset \adr(R)$. As $R$ is ideally ordered $\sub(R) \cong \mathsf{pi}(R)$ by Theorem \ref{T:Main} (a), and hence to show that $E_R\cong E_R^{\ADR}$ it is sufficient to show that $\adr(R) \subset \sub(R)$.

To show this consider an indecomposable object of $\adr(R)$. This is necessarily of the form $R e_i/\rad^l \, R e_i$ for some primitive idempotent $e_i$ corresponding to a vertex $i \in Q$ and integer $l=1, \dots, m$. As $Q$ has no sources it follows that there exists a series of arrows $j_{m-l} \xrightarrow{a_{m-l}} \dots \xrightarrow{a_3} j_2 \xrightarrow{a_2} j_{1} \xrightarrow{a_1} i$ such that the path $a:=a_{m-l} \dots a_1$ induces a homomorphism $Re_i \xrightarrow{a} R e_{j_{m-l}}$ of indecomposable projective $R$-modules. By construction this has kernel $\rad^{l} R e_i$, and hence there is an inclusion $Re_i/\rad^l \, Re_i \rightarrow Re_{j_{m-l}}$. In particular $R e_i/\rad^l \, R e_i \in \sub(R)$, and hence $\adr(R) \subset \sub(R)$. Hence $E_R\cong E_R^{\ADR}$.

Whilst the layer functions defining the quasi-hereditary structures on $E_R$ and $E_R^{\ADR}$ are not identical in general, we claim that the corresponding orderings do induce the same standard modules and hence the same quasi-hereditary structure on $E_R^{\ADR} \cong E_R$. To show this we let $P_{i,l}$ denote the projective $E_R \cong E_R^{\ADR}$-module $P(Re_i/\rad^l Re_i)$ and $S_{i,l}$ denote its simple quotient. We recall the order for $E_R^{\ADR}$ is defined by $S_{i,l}<S_{j,k} \Leftrightarrow  l>k$ and the order for $E_R$ is defined by $S_{i,l}<S_{j,k} \Leftrightarrow \dim Re_i/\rad^l Re_i > \dim Re_j/\rad^k Re_j$. In particular, both orderings induce strongly quasi-hereditary structures, and hence  for both orderings there are short exact sequences defining the respective standard modules 
\begin{align}\label{E:Standard}
0 \rightarrow \bigoplus P_{j,k} \rightarrow P_{i,l} \rightarrow \Delta(Re_i/\rad^l Re_i) \rightarrow 0
\end{align}
for each projective module $P_{i,l}$, see Definition \ref{D:RingelStrongQH}. Hence to show that the two orderings induce the same quasi-hereditary structure it is sufficient to show that the projective submodules $P_{j,k}$ of  $P_{i,l}$ appearing in \eqref{E:Standard} are the same for both orderings. For this we note that under the additive anti-equivalence 
\[
\Hom_{R}(-, \PI(R)):\mathsf{pi}(R) \rightarrow E_R \textnormal{-proj}
\]
an $E_R$-module $P_{j,k}$ is a proper submodule of $P_{i,l}$ if and only if the corresponding $R$-module $Re_{j}/\rad^k R e_j$ is a proper quotient of $Re_{i}/\rad^l R e_i$. This in turn is equivalent to $\dim Re_j/\rad^k R e_j < \dim Re_{i}/ \rad^{l} Re_{j'}$ and is also equivalent to $i=j$ and $k<l$.
This shows that the two orderings induce the same quasi-hereditary structure.
\end{proof}

It is a natural question whether it is possible to find an expanded class of algebras with a more general Ringel duality formula that encompasses both Theorems \ref{T:Main} and \ref{T:ADR}.

\subsection{Nilpotent quiver algebras} \label{Section:NQA}
The nilpotent quiver algebras introduced by Eiriksson \& Sauter \cite[Section 3]{EirikssonSauter} are a class of quasi-hereditary algebras.

\begin{defn}
Let $Q=(Q_0,Q_1)$ be a finite quiver. For $s\in \mathbb{Z}_{>0}$ the \emph{nilpotent quiver algebra} is defined to be 
\[
N_s(Q):= kQ^{(s)}/J
\]
where $Q^{(s)}$ is the \emph{staircase quiver} $Q^{(s)}$ defined by having vertices 
\[
i_l \text{ for } i \in Q_0 \text{ and } l\in \{1, \dots, s \}
\]
and arrows 
\begin{align*}
&b(i_l):i_{l+1} \rightarrow i_{l} &\text{ for }i \in Q_0 \text{ and }l \in \{1, \dots,s-1 \} \text{ and} \\
&a_l: h(a)_{l-1} \rightarrow t(a)_{l} &\text{ for } a \in Q_1 \text{ and } l \in \{2, \dots, s\},
\end{align*}
and where $J \subset kQ^{(s)}$ is the two-sided ideal generated by the relations
\begin{align*}
b(t(a)_l) a_{l+1}&=a_l b(h(a)_{l-1}) &\text{ for all } a \in Q_1 \text{ and } l \in \{2, \dots, s\}, \text{ and} \\
b(t(a)_1) a_2 &=0 &\text{ for all } a \in Q_1.
\end{align*}
\end{defn}

\begin{rem} We remark again that the nilpotent quiver algebra defined here is the opposite algebra of the nilpotent quiver algebra defined by Eiriksson and Sauter in \cite{EirikssonSauter}, however the effect on the quasi-hereditary structure is straightforward as is explained in Remark \ref{R:QH opposite algebras}.
\end{rem}

It follows from \cite[Proposition 3.15]{EirikssonSauter} that all nilpotent quiver algebras $N_s(Q)$ are right strongly quasi-hereditary and left ultra strongly quasi-hereditary for the quasi-hereditary structure determined by the layer function $L(i_t)=s-t$.

In particular, for $R=kQ/J^m$ the ADR and nilpotent quiver algebras are related as follows.

\begin{prop} \label{P:ADR and NSQ}
Let $Q$ be a finite quiver, $J$ the two-sided ideal generated by all arrows in $Q$, and $m$ a positive integer. Then there is an isomorphism of quasi-hereditary algebras
\[
N_{m}(Q) \cong  E_{kQ/J^m}^{\ADR} 
\]
if and only if all projective $kQ/J^m$-modules have Loewy length $m$: i.e. $Q$ contains no sinks and $m$ is arbitrary or $m=1$ and $Q$ is arbitrary.
\end{prop}

\begin{proof}
Let $R=kQ/J^m$, and let $e_i \in R$ for $i \in Q_0$ denote the primitive idempotents corresponding to vertices of $Q$. Up to isomorphism, the indecomposable modules in $\adr(R)$ are exactly $Re_i/ \rad^l Re_i$ for $1 \le l \le L(Re_i)$ and $i \in Q_0$, where $L(Re_i)$ is the Loewy length of the projective $Re_i$. 

In particular, the maximal Loewy length of a projective module in $R:=kQ/J^m$ is $m$ and so the maximum possible number of non-isomorphic indecomposables in $\adr(R)$ is $m|Q_0|$. But $|Q^{(m)}_0|=m|Q_0|$, so for $E^{\ADR}_R$ to be isomorphic to $N_{m}(Q)$ it is necessary that all projective $R$-modules have Loewy length $m$. 

Now suppose that all projective $R$-modules do have Loewy length $m$ and consider the algebra $E_R^{\ADR} :=\End_R(\ADR)$. We start by labelling the indecomposable module in $\adr(R)$ corresponding to $Re_i/ \rad^{l} \, Re_i$ by $i_l$ and hence label the corresponding primitive idempotent by $e_{i_l}$. There are indecomposable modules $i_{l}$ for $i \in Q_0$ and $l \in \{1,\dots, m \}$, matching the definition of the vertices in the staircase quiver $Q^{(m)}$.

We now want to produce a morphism $N_m(Q) \rightarrow E_R^{\ADR}$, and to do this we consider the morphisms between the indecomposable modules in $\adr(R)$.  Firstly, there are surjections $Re_i/ \rad^{l+1} Re_i \rightarrow Re_i/ \rad^{l} Re_i$ which we label by arrows $\beta(i_l):i_{l+1} \rightarrow i_{l}$ for $i \in Q_0$ and $l \in \{1, \dots, m-1 \}$. 

Secondly, an arrow $a: i \rightarrow j \in Q_1$ corresponds to a morphism of projectives $a: R e_j\rightarrow R e_i$ and for each $l$ this induces a morphism $Re_j \rightarrow Re_i / \rad^l \, Re_i$ with kernel $\rad^{l-1} \, Re_j$ which  in turn induces an injective morphism
\[
Re_j/ \rad^{l-1} Re_j \rightarrow Re_i/ \rad^{l} Re_i 
\]
for each $l \in \{2, \dots, m\}$. We label these morphisms by $\rho(a)_l:h(a)_{l-1} \rightarrow t(a)_{l}$ for $a \in Q_1$ and $l \in \{2, \dots, m\}$. In particular, the morphisms described here match the arrows of the staircase quiver $Q^{(m)}$ under the identification $a_l=\rho(a)_l$ and $b(i_l)=\beta(i_l)$.  In particular, an arrow $a:i \rightarrow j$ in $Q$ corresponds to a morphism $Re_j \rightarrow Re_i$ which induces morphisms
\[
\begin{tikzpicture}
\node (A0) at (0,0) {$j_{m}$};
\node (A1) at (1.5,0)  {$j_{m-1}$};
\node (A2) at (3,0)  {$\dots$};
\node (A3) at (4.5,0)  {$j_{l+1}$};
\node (A4) at (6,0)  {$j_{l}$};
\node (A5) at (7.5,0)  {$j_{l-1}$};
\node (A6) at (9,0)  {$\dots$};
\node (A7) at (10.5,0)  {$j_2$};
\node (A8) at (12,0)  {$j_1$};

\node (B0) at (0,2) {$i_{m}$};
\node (B1) at (1.5,2)  {$i_{m-1}$};
\node (B2) at (3,2)  {$\dots$};
\node (B3) at (4.5,2)  {$i_{l+1}$};
\node (B4) at (6,2)  {$i_{l}$};
\node (B5) at (7.5,2)  {$i_{l-1}$};
\node (B6) at (9,2)  {$\dots$};
\node (B7) at (10.5,2)  {$i_2$};
\node (B8) at (12,2)  {$i_1$};

\node (Dots1) at (3,1)  {$\dots$};
\node (Dots2) at (9,1)  {$\dots$};

\draw [left hook->] (A1) to node[gap] {$\scriptstyle{\rho(a)_{m}}$} (B0);
\draw [left hook->] (A4) to node[gap] {$\scriptstyle{\rho(a)_{l+1}}$} (B3);
\draw [left hook->] (A5) to node[gap] {$\scriptstyle{\rho(a)_{l}}$} (B4);
\draw [left hook->] (A8) to node[gap] {$\scriptstyle{\rho(a)_{2}}$} (B7);

\draw [->>] (A0) to node[above] {$\scriptstyle{\beta(j_{m-1})}$} (A1);
\draw [->>] (A1) to node[above] {$\scriptstyle{}$} (A2);
\draw [->>] (A2) to node[above] {$\scriptstyle{}$} (A3);
\draw [->>] (A3) to node[above] {$\scriptstyle{\beta(j_l)}$} (A4);
\draw [->>] (A4) to node[above] {$\scriptstyle{\beta(j_{l-1})}$} (A5);
\draw [->>] (A5) to node[above] {$\scriptstyle{}$} (A6);
\draw [->>] (A6) to node[above] {$\scriptstyle{}$} (A7);
\draw [->>] (A7) to node[above] {$\scriptstyle{\beta(j_{1})}$} (A8);

\draw [->>] (B0) to node[above] {$\scriptstyle{\beta(i_{m-1})}$} (B1);
\draw [->>] (B1) to node[above] {$\scriptstyle{}$} (B2);
\draw [->>] (B2) to node[above] {$\scriptstyle{}$} (B3);
\draw [->>] (B3) to node[above] {$\scriptstyle{\beta(i_l)}$} (B4);
\draw [->>] (B4) to node[above] {$\scriptstyle{\beta(i_{l-1})}$} (B5);
\draw [->>] (B5) to node[above] {$\scriptstyle{}$} (B6);
\draw [->>] (B6) to node[above] {$\scriptstyle{}$} (B7);
\draw [->>] (B7) to node[above] {$\scriptstyle{\beta(i_{1})}$} (B8);
\end{tikzpicture}
\]
where the relations $\beta(i_l) \rho(a)_{l+1}  \cong   \rho(a)_l \beta(j_{l-1}) $ and $\beta(i_1) \rho(a)_2\cong 0$ hold.

This allows us to define a morphism from the path algebra of the staircase algebra $ kQ^{(m)}$ to $E_R^{\ADR}$
by
\[
e_{i_t}  \mapsto e_{i_t}, \quad
b(i_t)  \mapsto \beta(i_t), \quad \text{ and }
a_l   \mapsto \rho(a)_l
\]
and, as the relations imposed on $kQ^{(m)}$ by $N_m(Q)$ are mapped to 0, this induces a morphism 
\[
\Phi:N_{m}(Q) \rightarrow E_R^{\ADR}.
\]
We will now show that $\Phi$ is surjective, and then calculate the dimensions of $N_m(Q)$ and $E_R^{\ADR}$ to show that it is an isomorphism.

Suppose that $f \in E_R^{\ADR} :=\End_R(\ADR)$ is a morphism 
\[
f:Re_i/\rad^l Re_i \rightarrow Re_j/\rad^k  Re_j
\]
for some $i,j \in Q_0$ and $l,k \in \{ 1, \dots, m\}$.
There is a surjection $\pi_{i,l}:Re_i \rightarrow Re_i /  \rad^l  Re_i$ and so $f$ gives a morphism $f \circ \pi_{i,l}: R e_i \rightarrow  Re_j/\rad^k  Re_j$. There is also a surjection $\pi_{j,k}: R e_j \rightarrow  Re_j/\rad^k  Re_j$, and as $Re_i$ is projective this induces a uniquely defined morphism $g:Re_i \rightarrow Re_j$ such that 
\[
\pi_{j,k} \circ g \cong f \circ \pi_{i,l}.
\] 
As a morphism between projective modules, the morphism $g: Re_i \rightarrow R e_j$ corresponds to an element $g \in e_i R e_j \subset R \cong \End_R(\bigoplus R e_i)^{\op}$. In particular, $g= \sum_p \lambda_p p \in kQ/J^m=R$ for scalars $\lambda_p$ and homogeneous paths $p$ from $j$ to $i$ in $kQ/J^m=R$ corresponding to morphisms $p:Re_i \rightarrow Re_j$. 

We now work with one indecomposable path $p$, corresponding to a morphism $p:Re_i \rightarrow Re_j$, and suppose that $p$ consists of $n:=|p|$ arrows $p=a_n a_{n-1} \dots a_1$ for $a_i \in Q_1$. We define a corresponding path in $N_m(Q)$ from $i_{l-n}$ to $j_l$ by
\[
(p)_l:=(a_1)_{l} \dots (a_{n-1})_{l+1}(a_n)_{l-n+1}: i_{l-n} \rightarrow j_l
\]
for $l \in \{n+1, \dots, m\}$. Similarly, we define the path in $N_m(Q)$
\[
(\pi_{i,l}):= b(i_l) b(i_{l+1}) \dots b(i_{m-1}): i_m \rightarrow i_l
\]
for $i \in Q_0$ and $l \in \{1, \dots, m\}$ such that $\Phi( (\pi_{i,l}) ) \cong \pi_{i,l}$. Then the morphism $p$ factors over its kernel, which is $\rad^{m-n} \, R e_i$, so 
\[
p \cong  \Phi \left( (p)_m (\pi_{i,m-n}) \right).
\]
Hence
\[
\pi_{j,k} \circ p \cong  \Phi \left(  (\pi_{j,k}) (p)_m (\pi_{i,m-n}) \right).
\]
Using the relations in $N_m(Q)$ we can rearrange this expression as 
\begin{align*}
(\pi_{j,k}) (p)_m (\pi_{i,m-n}) &= b(j_k) \dots b(j_{m-1}) (p)_m (\pi_{i,m-n})  \\
&= (p)_k b(i_{k-n}) \dots b(i_{m-n-1}) (\pi_{i,m-n}) \\
 &= (p)_k (\pi_{i,k-n}) 
 \end{align*}
 where we note that if $g$ is non-zero then $k-n  \le l$ and hence
 \[
 (\pi_{j,k}) (p)_m (\pi_{i,m-n})= (p)_k (\pi_{i,k-n}) = (p)_k b(i_{k-n}) \dots b(i_{l-1})(\pi_{i,l}),
 \]
 and hence 
 \[
 \Phi((p)_k b(i_{k-n}) \dots b(i_{l-1})(\pi_{i,l}) ) \cong \pi_{j,k} \circ p.
 \]
Returning to the morphism $g = \sum \lambda_p p$ we see that
\[
\sum_p \lambda_p \pi_{j,k}  \circ p  \cong  \pi_{j,k}  \circ \sum_p \lambda_p p \cong \pi_{j,k} \circ g \cong f \circ \pi_{i,l}
\]
and we can now conclude that 
\[
\sum_p \lambda_p \Phi( (p)_k b(i_{k-|p|}) \dots b(i_{l-1})(\pi_{i,l}) ) \cong  \Phi \left( \sum \lambda_p (p)_k b(i_{k-|p|}) \dots b(i_{l-1}) \right) \circ \pi_{i,l} \cong f \circ \pi_{i,l}
\]
where $|p|$ is the length of the path p, but $\pi_{i,l}$ is surjective and hence 
\[
f \cong \Phi(\sum_p \lambda_p (p)_{k} b(i_{k-|p|}) \dots b(i_{l-1})).
\]
We conclude that $\Phi$ is a surjection, and we now show that this surjective morphism is in fact an isomorphism by calculating the dimensions of $N_m(Q)$ and $E_R^{\ADR}$.

We first calculate the dimension of $E_R^{\ADR}$ by calculating the dimension of the morphisms between any two indecomposables in $\adr(R)$. As shown above, a morphism in $E_R^{\ADR}$ of the form $f: R e_i/  \rad^l  Re_i \rightarrow Re_j/ \rad^k  Re_j$ is induced by a particular element in $kQ/J^m$ corresponding to a morphism of projective modules $Re_i \rightarrow Re_j$. Such elements are spanned by the paths, and we now calculate the morphisms in $E_R^{\ADR}$ that are induced by such path in $R=kQ/J^m$. These will give a basis for the morphisms $Re_i/ \rad^l Re_i \rightarrow Re_j / \rad^k Re_j$. A path $p:j \rightarrow i \in kQ/J^m = R$ of length $|p|$ (under the length grading on $Q$) induces the morphism $p:Re_i \rightarrow Re_j$ which composes to give a nonzero morphism $Re_i \rightarrow R e_j / \rad^k  Re_j$ if and only if $|p|<k$. In turn, this descends to give a nonzero morphism $Re_i/ \rad^l  Re_i \rightarrow R e_j / \rad^k  Re_j$ if and only if $\rad^l  Re_i \subset \ker(p)=\rad^{k-|p|}  R e_i$, which occurs if and only if $l \ge k-|p|$. As such there are isomorphisms of vector spaces
\begin{align*}
e_{j_k} E_{R}^{\ADR}  e_{i_l}  &= \left\langle \begin{array}{c} \text{Elements of $\Hom_R(Re_i,Re_j)$ that factor through} \\
\text{$\Hom_R(Re_i/ \rad^l \, Re_i,Re_j/ \rad^k \, Re_j)$. } \end{array} \right \rangle \\
&=  \left \langle \begin{array}{c}
\text{Paths }\, p \in e_i R e_j  \\
\text{ such that } k-l  \le |p| < k.
\end{array}
\right \rangle.
\end{align*}

We then calculate the dimension of $N_m(Q)$ by counting the number of paths between any two vertices. Using the explicit description of $N_m(Q)$ above, any path in $N_m(Q)$ corresponds to the composition of arrows of type $a_l$ and arrows of type $b(i_t)$, these commute $b(t(a)_l)a_{l+1}=a_l b(h(a)_{l-1})$, and $b(t(a)_1) a_2=0$. Using these relations any non-zero path can be rearranged such that all the $b(i_t)$ type arrows occur in the path before the $a_l$ type arrows. That is: a path from $i_l$ to $j_k$ in $N_m(Q)$ exactly corresponds to the path $(a_1)_{k} \dots (a_n)_{k-|p|+1}$ in $N_m(Q)$ induced by a path $p=a_{|p|} \dots a_1$ from $j$ to $i$ in $Q$ of length $|p|$ pre-composed with $l-k+|p|$ arrows of $b(i_t)$ type 
\[
(a_1)_{k} \dots (a_n)_{k-|p|+1} b(i_{k-|p|}) \dots b(i_{l-1}): i_l \rightarrow j_k
\]
so that the induced path is from $i_l$ to $j_k$. However, the path is nonzero if and only if the number of type $b(i_t)$ arrows is greater than or equal to $0$ and strictly less than $l$, and it follows that 
\begin{align*}
 e_{j_k} N_m(Q)  e_{i_l} & = \left \langle \begin{array}{c}
\text{ Paths $p \in e_i R e_j $} \\
\text{ such that $ 0 \le l-k+|p| < l $}
\end{array}
\right \rangle.
\end{align*}
Hence
\begin{align*}
\dim e_{j_k} N_m(Q) e_{i_l}  
 =  \dim \langle \, p \in e_i R e_j \, | \, k-l  \le |p| < k \rangle = \dim e_{j_k}  E_R^{\ADR} e_{i_l}.
\end{align*}
It follows that the surjective homomorphism $\Phi:N_m(Q) \twoheadrightarrow E_R^{\ADR}$ is in fact an isomorphism as $\dim N_m(Q)=\dim E_R^{\ADR}$. Hence $E_{R}^{\ADR} \cong N_{m}(Q)$.

Further, under this isomorphism the layer functions defining the quasi-hereditary structures on $N_m(Q)$ and $E_R^{\ADR}$ are identified and hence this is an isomorphism of quasi-hereditary algebras.
\end{proof}

 \begin{ex}
 We give a brief example of Proposition \ref{P:ADR and NSQ}. Consider the quiver 
 \[
\begin{tikzpicture} []
\node (c) at (-1.3,0) {$Q:=$};
\node (C1) at (0,0)  {$1$};
\node (C2) at (1.5,0)  {$2$};
\draw [->] (C2) to node[above]  {\scriptsize{$a$}} (C1);
\draw [->, looseness=18, in=142, out=218,loop] (C1) to node[right] {\scriptsize{$x$}} (C1);
\end{tikzpicture}
\]
and let $J$ denote the two-sided ideal generated by all arrows. Define $R:=\mathbb{C}Q/J^3$, and then we present the two algebras $N_3(Q)$ and $E_{R}^{\ADR}$.

Firstly, the  algebra $N_3(Q)$ is defined to be the path algebra of the quiver with relations
\begin{align*}
\begin{aligned}
\begin{tikzpicture} []
\node (c) at (-1.3,0) {$Q^{(3)}:=$};
\node (C13) at (0,0)  {$1_3$};
\node (C12) at (0,-1)  {$1_2$};
\node (C11) at (0,-2)  {$1_1$};
\node (C23) at (1.5,0)  {$2_3$};
\node (C22) at (1.5,-1)  {$2_2$};
\node (C21) at (1.5,-2)  {$2_1$};
\draw [->] (C11) to node[above]  {\scriptsize{$a_2$}} (C22);
\draw [->] (C12) to node[above]  {\scriptsize{$a_3$}} (C23);
\draw [->, bend right = 25] (C11) to node[right] {\scriptsize{$x_2$}} (C12);
\draw [->, bend right = 25] (C12) to node[right] {\scriptsize{$x_3$}} (C13);
\draw [->, bend right = 25] (C12) to node[left] {\scriptsize{$b(1_1)$}} (C11);
\draw [->, bend right = 25] (C13) to node[left] {\scriptsize{$b(1_2)$}} (C12);
\draw [->, bend left = 25] (C23) to node[right] {\scriptsize{$b(2_2)$}} (C22);
\draw [->, bend left = 25] (C22) to node[right] {\scriptsize{$b(2_1)$}} (C21);
\end{tikzpicture}
\end{aligned}
\qquad
\begin{aligned}
b(1_1) x_2&=0, \\
b(2_1)a_2 &= 0, \\
b(1_2) x_3&=x_2 b(1_1), \text{ and } \\
b(2_2)a_3 &= a_2b(1_1).
\end{aligned}
\end{align*}

Secondly, we consider the indecomposable modules in $\adr(R)$. There are 6 classes and we list them and a basis for all injective or surjective maps between them below.
\begin{align*}
\begin{aligned}
1_3:&=Re_1=\langle e_1, x, x^2 \rangle, & 2_3:&=Re_2=\langle e_2, a, xa \rangle, \\
1_2:&=Re_1/\rad^2 Re_1=\langle e_1, x \rangle, & 2_2:&=Re_2/\rad^2 Re_2=\langle e_2, a \rangle,  \\
1_1:&=Re_1/\rad^1 Re_1=\langle e_1\rangle, &2_1:&=Re_2/\rad^1 Re_2=\langle e_2\rangle.
\end{aligned}
\quad
\begin{aligned}
\begin{tikzpicture} []
\node (C13) at (0,0)  {$1_3$};
\node (C12) at (0,-1)  {$1_2$};
\node (C11) at (0,-2)  {$1_1$};
\node (C23) at (1.5,0)  {$2_3$};
\node (C22) at (1.5,-1)  {$2_2$};
\node (C21) at (1.5,-2)  {$2_1$};
\draw [right hook->] (C11) to node[above]  {\scriptsize{$a$}} (C22);
\draw [right hook->] (C12) to node[above]  {\scriptsize{$a$}} (C23);
\draw [right hook->, bend right = 15] (C11) to node[right] {\scriptsize{$x$}} (C12);
\draw [right hook->, bend right = 15] (C12) to node[right] {\scriptsize{$x$}} (C13);
\draw [->>,bend right = 15] (C12) to node[left] {} (C11);
\draw [->>,bend right = 15] (C13) to node[left] {} (C12);
\draw [->>] (C23) to node[right] {} (C22);
\draw [->>] (C22) to node[right] {} (C21);
\end{tikzpicture}
\end{aligned}
\end{align*}
This describes $E_R^{\ADR} :=\End_R(\ADR)$ and matches the path algebra with relations description of $N_3(Q)$ above.
 \end{ex}

Combining Proposition \ref{P:ADR and NSQ} with Proposition \ref{P:ADR=IOM} and Theorem \ref{T:Main} (or Theorem \ref{T:ADR}) instantly gives the following corollary.
\begin{cor} \label{C:Qnosinksnosources}
If $Q$ is a finite quiver without sinks or sources and $m$ is a positive integer, then there are isomorphisms of quasi-hereditary algebras
\[
E_{kQ/J^m} \cong E_{kQ/J^m}^{\ADR} \cong N_{m}(Q).
\]
In particular, the Ringel dual for a nilpotent quiver algebra without sinks or sources is determined by the formula
\[
\mathfrak{R}(N_m(Q)) \cong N_m(Q^\op)^{\op}.
\]
\end{cor}
We note that if $Q$ is a finite quiver with no sinks but with sources then $E_{kQ/J^m}^{ADR} \cong N_m(Q)$ but $E_{kQ^{\op}/J^m}^{\ADR} \not \cong N_m(Q^\op)^{\op}$ (if $m>1$) as $kQ^{\op}$ contains sinks. In particular, Proposition \ref{P:ADR and NSQ} and Theorem \ref{T:ADR} cannot be used to strengthen the Ringel duality formula of Corollary \ref{C:Qnosinksnosources} to all quivers with no sources.

\subsection{Auslander and Nakayama algebras} \label{Section:AuslanderNakayama}

For a finite dimensional algebra $R$ of finite representation type we define $\AUS:= \bigoplus_{M \in \ind (R-\mod)} M$, where the sum is taken over all indecomposable $M \in R-\mod$ up to isomorphism,  and the Auslander algebra is defined to be
\[
E_R^{\AUS}:= \End_R(\AUS).
\]

\begin{prop} \label{p:AUSisIOM} If $R$ is an ideally ordered monomial algebra, then
\[
E_R^{\AUS} \cong E_R
\]
if and only if $R$ is self-injective.
\end{prop}
\begin{proof}

If $E_R^{\AUS} \cong E_R$ then $R-\mod \cong \sub(R)$, hence every injective $R$-module $I$ embeds into $R^n$. Therefore, $I$ is a direct summand of $R^n$, hence projective, and hence $R$ is self-injective. 

Conversely, if $R$ is self-injective, then every injective $R$-module embeds into $R^n$ for some $n$ and hence every injective module is also a projective module. Then every object in $R-\mod$ is a submodule of an injective $R$-module, hence of a projective $R$-module, hence $R-\mod \cong \sub(R)$ and $E_R^{\AUS} \cong E_R$.
\end{proof}

The Nakayama algebras, introduced in \cite{Nakayama}, are a well known class of finite dimensional algebras with finite representation type; see e.g. \cite[Theorem VI.2.1]{ARS}. Recall that a self-injective Nakayama algebra is of the form $kC_n/J^m$ where $C_n$ is an oriented cycle with $n$ vertices and $J$ is the ideal generated by all arrows, see e.g. \cite[Theorem 32.4]{AndersonFuller} for the a description of the underlying quiver of a general Nakayama algebra. In particular, the self injective Nakayama algebras are ideally ordered monomial algebras. 
 
\begin{cor}\label{c:AUSisNAK} If $R$ is a self-injective Nakayama algebra, then
\[
E_R^{\AUS} \cong E_R.
\]
\end{cor}
It follows from the explicit description $R=kC_n/J^m$ that $E_R=N_m(C_n)$ by Corollary \ref{C:Qnosinksnosources} and so this corollary recovers the well-known explicit description of the Auslander algebras of self-injective Nakayama algebra $E_R$ in terms of quivers with relations. As $(kC_n/J^m)^{\op} \cong kC_n/J^m$ and $N_m(C_n) \cong N_m(C_n^{\op})^{\op}$ the Ringel duality formula  recovers the result of \cite{Tan} that the Auslander algebras of self-injective Nakayama algebras are Ringel-self-dual for the ideal layer function.

\begin{cor} \label{c:RingelSelfDual}
For a self-injective Nakayama algebra $R$
\[
\mathfrak{R}(E_R) \cong E_R.
\]
\end{cor}

\begin{rem}
In order to give another perspective on Proposition \ref{p:AUSisIOM}, Corollary \ref{c:AUSisNAK} and Corollary \ref{c:RingelSelfDual}, we recall that self-injective finite dimensional monomial algebras $R$ are Nakayama algebras. To see this, we have to show that the quiver $Q$ underlying $R$ is a union of oriented lines and oriented cycles. In other words, at every vertex of $Q$ there is at most one incoming and at most one outgoing arrow. Assume that there is a vertex $i$ with more than one outgoing arrow. Then, as $R$ is monomial, the corresponding indecomposable projective $R$-module $P_i$ does not have a simple socle - in particular, $P_i$ is not injective contradicting our assumption that $R$ is self-injective. A dual argument shows that $Q$ does not have vertices with more than one incoming arrow.  
\end{rem}

\section{Appendix: Results on finite dimensional monomial algebras.}

In this section we collect some technical results on finite dimensional monomial algebras $R=kQ/I$ (where $I$ is generated by a collection of paths in $Q$). We will use the term `monomial' to mean a monomial expression in the generators (i.e. arrows and lazy paths) of such an algebra.

\begin{lem}\label{L:StandardEpi}
Let $R$ be a monomial algebra and $n, m \in R$ monomials. If there exists a surjection $\phi:Rm \to Rn$, then the map $Rm \to Rn$ defined by $m \mapsto n$ is $R$-linear.
\end{lem}

\begin{proof}
It suffices to show that $\mathrm{ann}_R(m)$ is contained in $\mathrm{ann}_R(n)$. 
Take $r \in R$ with $rm=0$, and we aim to show that $rn=0$. We write $r = \sum \lambda_i r_i$ with monomials $r_i$ and non-zero scalars $\lambda_i$. Since $R$ is monomial, it follows that $r_i m=0$ for all $i$. The existence of a surjection $\phi:Rm \to Rn$ implies $m, n \in eR$ for some primitive idempotent $e \in R$ and that there exist $s,t \in R$ such that $\phi(tm)=n$ and $\phi(m)=sn$. In particular, $tsn=n$ and so $s=\mu_0 e + \sum_{i=1}^t \mu_i s_i$ for some non-zero scalars $\mu_i$ and distinct monomials $s_i \neq e$. Therefore $r_i s n =\phi(r_im)=\phi(0)=0$, and so as $R$ is monomial it follows that all monomials that make up $r_i s n$ are 0. In particular, $r_i \mu_0 e n=\mu_0 r_i n =0$. This implies that $r_in=0$ for all $i$, and hence $rn=0$ so $\mathrm{ann}_R(m) \subset \mathrm{ann}_R(n)$ finishing the proof.
\end{proof}

\begin{lem}\label{L:GenToGen}
Let $m, n \in R$ be monomials. If $R$ is ideally ordered, then every surjection
$Rm \to Rn$ factors over $\pi \colon Rm \to Rn$, $m \mapsto n$.
\end{lem}

\begin{proof}
Let $\psi \colon Rm \to Rn$ be an surjection. In particular, $m, n \in eR$ for some primitive idempotent $e \in R$ and there exist $s,t \in R$ such that $\psi(m)=sn$ and $\psi(tm)=n$. It follows that $tsn=n$, so $s=\lambda_0 e + \sum \lambda_i s_i \in eRe$ for nonzero scalars $\lambda_i$ and distinct monomials $s_i \neq e$. Hence $sn=\lambda_0 n + \sum \lambda_i s_i n$. In particular, $R s_i n \subsetneq Rn$, and since $R$ is ideally ordered there exists surjections $Rn \to R s_i n$ which, using Lemma \ref{L:StandardEpi}, we can assume are defined by $n \mapsto s_i n$. Denote the composition of such a surjection with the inclusion $R s_i n \subseteq Rn$ by $\varphi_i$ and define $\varphi\colon Rn \to Rn$  
as $\varphi= \lambda_0 \mathrm{id} + \sum \lambda_i \varphi_i$. Then $\varphi(n)=sn$ and therefore $\psi = \varphi \pi$ factors as claimed. 
\end{proof}

\begin{lem}\label{L:PIareMonom}
Let $p \in eR$ for a primitive idempotent $e \in R$. If $R$ is ideally ordered, then the principal left ideal $Rp$ is isomorphic to a principal ideal $Rm$, for a monomial  $m \in eR$.
\end{lem}

\begin{proof}
Since $R$ is monomial, we may write $p$ as linear combination of monomials $p=\sum_{i=1}^t \lambda_i p_i$ with $\lambda_i$ non-zero scalars and $p_i \in eR$ monomials. Since $R$ is ideally ordered we may assume that the $p_i$ are labelled in such a way that $Rp_1 \to Rp_2 \to \cdots \to Rp_t$ are surjections. 

We now wish to rewrite $p$ so that none of the $p_i$ can be expressed in the form $np_1$ for a monomial $n$. To do this, let $I$ index the $p_i$ such that there is a monomial $r_{i}$ with $p_{i}=r_{i}p_1$ for $i \in I$. Then we define $s=\lambda_1 e + \sum_{i \in I} \lambda_{i} r_{i}$ and $p=sp_1+ \sum_{i \notin I} \lambda_i p_i$. As $r:=\sum_{i \in I} \lambda_{i} r_{i} \in \rad R \cap eRe$ it follows that $s= \lambda_1 e+r$ is a unit in $eRe$ and there exists $t \in eRe$ such that $st=e$. In particular, $Rtp =Rp$. Then we rewrite $tp=tsp_1+ \sum_{i\notin I} \lambda_i tp_i = p_1 + \sum_{j=2}^{t'} \mu_j q_j$ for some non-zero scalars $\mu_j$ and monomials $q_j \ne e$. For each $q_j$ there is some $p_i$ such that $Rq_j \subset Rp_i$ by their definition, and hence there are surjections $Rp_1 \rightarrow Rq_j$ for all $j$. As $Rtp \cong Rp$ we now work with $tp$ rather than $p$ and $tp$ has the property that there are no $q_j$ with $np_1=q_j$ for a monomial $n$.

We claim that $Rtp \cong Rp_1$, hence $Rp \cong Rp_1$. As there are surjections $Rp_1 \rightarrow Rq_j$ there are surjections $Rp_1 \to Rq_j$, $p_1 \mapsto q_j$ by Lemma \ref{L:StandardEpi}. Let $\varphi_j$ be the composition of such a surjection with the canonical inclusion $Rq_j \to R$ and let $\iota\colon Rp_1 \to R$ be the canonical inclusion. Define $\psi\colon Rp_1 \to R$ by $\psi=\iota +\sum_{j=2}^{t'} \mu_j \varphi_j$. Then $\psi(p_1)= p_1 +\sum_{j=2}^{t'} \mu_j q_j= tp$ so $\im \psi = Rtp$. Hence $\psi$ defines a surjective morphism $\phi:Rp_1 \rightarrow Rtp$.

We must now check that this morphism is also injective. If $\psi(rp_1)=0$, then $rp_1 + r\sum_{j=2}^{t'} \mu_j q_j =0 $. As $R$ is monomial if $rp_1$ is non-zero there must exist monomials $n,m \in R$ such that $np_1=mq_j$ for some $j$, and if this occurs either $p_1 = m'q_j$ or $n'p_1=q_j$ for submonomials $m'$ and $n'$ neither equal $e$. The first case cannot occur as this implies $R p_1 \ne \subset Rq_j$ which contradicts the existence of a surjection $Rp_1 \rightarrow Rq_j$. The second situation also cannot occur as the construction of the $q_j$ above ensured none were of this form. Hence $rp_1=0$ so the morphism is also injective and $Rp_1 \cong Rtp \cong Rp$.
\end{proof}

\end{document}